%% file: 2D2.tex
\documentclass[12pt]{article}

\input{preamble.tex}

\usepackage{xifthen}  % http://ctan.org/pkg/xifthen
\usepackage{breqn}
\usepackage{yfonts}
\usepackage{ulem}

\usepackage{fullpage}
\usepackage{pdflscape}

\usepackage[affil-it]{authblk}

\tikzstyle{shaded}=[fill=red!10!blue!20!gray!30!white]
\tikzstyle{unshaded}=[fill=white]
\tikzstyle{empty box}=[circle, draw, thick, fill=white, opaque, inner sep=2mm]
\tikzstyle{annular}=[scale=.7, inner sep=1mm, baseline]
\tikzstyle{rectangular}=[scale=.75, inner sep=1mm, baseline=-.1cm]

\usepackage{xcolor}
\definecolor{dark-red}{rgb}{0.7,0.25,0.25}
\definecolor{dark-blue}{rgb}{0.15,0.15,0.55}
\definecolor{medium-blue}{rgb}{0,0,0.65}
\hypersetup{
   colorlinks, linkcolor={purple},
   citecolor={medium-blue}, urlcolor={medium-blue}
}

\DeclareMathOperator{\spann}{span}
\DeclareMathOperator{\Tr}{Tr}
\DeclareMathOperator{\trains}{trains}

\newcommand{\jw}[1]{f^{(#1)}}

\renewcommand{\set}[2]{\left\{#1\middle|#2\right\}}

\newcommand{\nbox}[6]{
	\draw[thick, #1] ($#2+(-#3,-#3)+(-#4,0)$) rectangle ($#2+(#3,#3)+(#5,0)$);
	\coordinate (ZZa) at ($#2+(-#4,0)$);
	\coordinate (ZZb) at ($#2+(#5,0)$);
	\node at ($1/2*(ZZa)+1/2*(ZZb)$) {#6};
}

\newcommand{\ncircle}[5]{
	\draw[thick, #1] #2 circle (#3);
	\node at #2 {#5};
	\node at ($#2+(#4:.1cm)+(#4:#3cm)$) {\scriptsize{$\star$}};
}

\newcommand{\gA}{{\textgoth{A}}}

\title{2-supertransitive subfactors at index $3+\sqrt{5}$}
\author[1]{Kim Morrison}
\affil[1]{Mathematical Sciences Institute, Australian National University}

\author[2]{David Penneys}
\affil[2]{University of Toronto}
\begin{document}
\date{June 13, 2014}
\maketitle

\begin{abstract}
This article proves the existence and uniqueness of a subfactor planar algebra with principal graph consisting of a diamond with arms of length 2 at opposite sides, which we call ``2D2''.
We also prove the uniqueness of the subfactor planar algebra with principal graph 4442.
We conjecture this will complete the list of subfactor planar algebras at index $3+\sqrt{5}$.
% This is the published version of \arxiv{...}
\end{abstract}

%%%%%%%%%%%%%%%%%%%%%%%%%%%%%%%%%%%%%%%%%%%%%%%%%%%%%%%%%%%%%%%%%%%%%%%%%%%
\section{Introduction}

The small index subfactor classification program has seen considerable success, with the most recent progress to index 5 \cite{MR2914056,MR2902285,MR2993924,MR2902286}.
The survey \cite{MR3166042} provides a comprehensive overview of the series, along with a general introduction to the methods.

This classification, as with previous classifications to index 4 \cite{MR0696688,MR996454,MR1193933,MR999799,MR1278111,MR1213139} and $3+\sqrt{3}$ \cite{MR1317352,MR1686551,MR1625762,MR2472028,MR2979509}, has three main steps:
\begin{enumerate}[(1)]
\item
Enumerate families of potential graph pairs, organized as `vines', `weeds', and `cylinders'.
\item
Apply obstructions to reduce these families to finitely many graph pairs.
\item
Determine how many examples there are (if any!) for surviving graph pairs.
\end{enumerate}
The enumeration step uses special obstructions which are able to simultaneously rule out all translated extensions of a given graph (see \cite{MR2914056} for the notions of translation and extension).
Progress to higher indices has only been possible because of the discovery of new obstructions of this type.
For example, the enumeration step for the classification to index $5$ \cite{MR2914056} used Ocneanu's triple point obstruction and associativity constraint \cite{MR1317352}. 

Ongoing joint work of Afzaly-Morrison-Penneys has made substantial progress toward extending the classification to index $3+\sqrt{5}$.
There have been several breakthroughs that have made this further enumeration possible, but we will not discuss them here.
The purpose of the present article is to complete the third step described above for the remaining promising candidates from this enumeration step.
That is, we construct one example and prove some uniqueness and non-existence results.
Such examples are rare and difficult to construct; moreover, they are of great importance to not only subfactor theory, but also fusion categories and conformal field theory.

In particular, we prove the following main theorems.

\begin{thm}\label{thm:2D2}
There is a unique  subfactor planar algebra with principal graphs ``2D2''
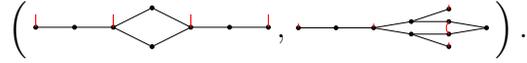
$$
\left(
\bigraph{bwd1v1v1p1v1x1v1v1duals1v1v1v1},
\bigraph{bwd1v1v1p1v1x0p1x0p0x1p0x1v0x1x1x0duals1v1v1x3x2x4}
\right).
$$
\end{thm}

\begin{thm}\label{thm:Unique}
There are no subfactor planar algebras with principal graphs
$$
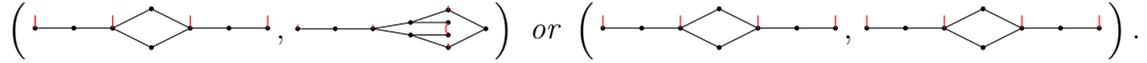
\left(
\bigraph{bwd1v1v1p1v1x1v1v1duals1v1v1v1},
\bigraph{bwd1v1v1p1v1x0p1x0p0x1p0x1v1x0x0x1duals1v1v1x3x2x4}
\right)
\text{ or }
\left(
\bigraph{bwd1v1v1p1v1x1v1v1duals1v1v1v1},
\bigraph{bwd1v1v1p1v1x1v1v1duals1v1v1v1}
\right).
$$
\end{thm}

We also prove the following, which was also known to Izumi.

\begin{thm}\label{thm:4442}
There is exactly one subfactor planar algebra with principal graphs 
$$
4442=
\left(\bigraph{bwd1v1v1v1v1p1p1v1x0x0p0x1x0p0x0x1v1x0x0p0x1x0v1x0p0x1duals1v1v1v2x1x3v2x1},\bigraph{bwd1v1v1v1v1p1p1v1x0x0p0x1x0p0x0x1v1x0x0p0x1x0v1x0p0x1duals1v1v1v2x1x3v2x1}\right)
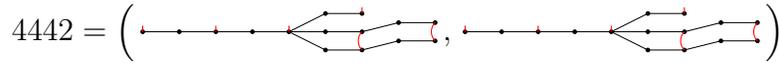
$$
(which was constructed in \cite{1208.3637}).
\end{thm}

Recall the conjecture of Morrison-Peters \cite{1205.2742}:

\begin{conj}[{\cite[Conjecture 2.2]{1205.2742}}]\label{conj:ToIndex3PlusSqrt5}
There are exactly 2 non-Temperley-Lieb subfactor planar algebras with index in $(5,3+\sqrt{5})$:
\begin{itemize}
\item
the unique $\mathfrak{su}(2)_5$ subfactor planar algebra \cite{MR936086,1205.2742}, and
\item
the unique $\mathfrak{su}(3)_4$ subfactor planar algebra \cite{MR936086,1205.2742}.
\end{itemize}
\end{conj}

\begin{remark}
Conjecture \ref{conj:ToIndex3PlusSqrt5} is known to hold when restricted to exactly 1-supertransitive subfactor planar algebras \cite[Theorem 2.1]{1205.2742}.
In fact, the recent article \cite{1310.8566}, using techniques of \cite{1308.5656}, has classified all exactly 1-supertransitive subfactor planar algebras without intermediates with index at most $6\frac{1}{5}$.
Moreover, \cite{1308.5691} (see also \cite{1308.5723}) has classified exactly 1-supertransitive subfactor planar algebras at index $3+\sqrt{5}$.
\end{remark}

We further conjecture that we now have a complete list of the possible subfactor planar algebras with index in $(5,3+\sqrt{5}]$.

\begin{conj}
At index $3+\sqrt{5}$, there are exactly 13 non-Temperley-Lieb subfactor planar algebras:
\begin{itemize}
\item
the Bisch-Jones Fuss-Catalan $A_3*A_4$ subfactor planar algebra and its dual \cite{MR1437496},
\item
the 5 quotients of the Fuss-Catalan $A_3*A_4$ subfactor planar algebra \cite{1308.5691,1308.5723} (including the self-dual tensor product),
\item
the unique 2D2 subfactor planar algebra and its dual,
\item
Izumi's unique symmetrically self-dual $3^{\bbZ/2\bbZ\times\bbZ/2\bbZ}$ subfactor planar algebra \cite{IzumiUnpublished,1208.3637},
\item
Izumi's unique $3^{\bbZ/4\bbZ}$ subfactor planar algebra and its dual \cite{IzumiUnpublished,1308.5197}, and
\item
the unique symmetrically self-dual 4442 subfactor planar algebra \cite{1208.3637}.
\end{itemize}
\end{conj}

Izumi proved the uniqueness of the $3^{\bbZ/2\bbZ \times\bbZ/2\bbZ}$ and $3^{\bbZ/4\bbZ}$ subfactor planar algebras in unpublished notes by finding all Cuntz algebra models of such subfactors \cite{IzumiUnpublished}.
Hence to prove the above conjectures, it remains to show that there are no other subfactor planar algebras with index at most $3+\sqrt{5}$.
This is precisely the goal of the work of Afzaly-Morrison-Penneys.

We remark that these unpublished notes of Izumi \cite{IzumiUnpublished} also prove the existence of a subfactor with principal graph
$$
2D2=\bigraph{bwd1v1v1p1v1x1v1v1duals1v1v1v1}.
$$
The possible graph partners of $2D2$, subject to the associativity condition \cite{MR1317352,MR2914056}, are
$$
\bigraph{bwd1v1v1p1v1x1v1v1duals1v1v1v1},\,
\bigraph{bwd1v1v1p1v1x0p1x0p0x1p0x1v1x0x0x1duals1v1v1x3x2x4},\text{ and }
\bigraph{bwd1v1v1p1v1x0p1x0p0x1p0x1v0x1x1x0duals1v1v1x3x2x4},
$$
and each of these graph pairs has a bi-unitary connection.
As showing flatness of a connection \cite{MR996454,MR1642584} can be extremely demanding, we use planar algebra techniques to prove Theorems \ref{thm:2D2} and \ref{thm:Unique} simultaneously.

To construct the 2D2 subfactor planar algebra, we follow Jones' program \cite{MR1929335} of finding it inside a graph planar algebra \cite{MR1865703}.
In the graph planar algebra of 2D2, we find a one-parameter family (up to sign) of uncappable rotational eigenvectors $T=T(\lambda)$ for $\lambda\in\bbT$ which satisfy $T^2=\jw{3}$. 
We then derive relations for this generator, and we prove that independent of $\lambda\in \bbT$, $T$ generates the 2D2 subfactor planar algebra.
Not only does this prove existence and uniqueness for 2D2 in Theorem \ref{thm:2D2}, it also proves nonexistence of the graph pairs in Theorem \ref{thm:Unique} (see Section \ref{sec:PrincipalGraphs} for more details).

By \cite{MR3157990}, Bigelow's jellyfish algorithm \cite{MR2577673,MR2979509} is universal for finite depth subfactor planar algebras.
We make this statement precise in Theorem \ref{thm:EvaluableSubalgebra} and Remark \ref{rem:FiniteDepth}.
We apply this to 2D2 as follows.
We use $T$ to find a system of matrix units $\cS_{4,+}$ for a copy of $M_2(\bbC)$ corresponding to the vertex at depth 4 of 2D2.
Since 2D2 is stable in the sense of \cite{MR1334479} at depth 4, we can show that 2-strand jellyfish relations hold for the elements of $\cS_{4,+}$, in Theorem \ref{thm:2Strand}. These are sufficient to evaluate closed diagrams by Theorem \ref{thm:Evaluable}.
This shows that the planar subalgebra generated by $T$ in the graph planar algebra of 2D2 is a subfactor planar algebra.
Interestingly, the universal jellyfish algorithm described in Theorem \ref{thm:EvaluableSubalgebra} uses generators which are not necessarily uncappable rotational eigenvectors.

The even half of the 2D2 subfactor planar algebra belongs to the family of generalized Haagerup categories, which Grossman-Izumi-Snyder have recently used to connect the Asaeda-Haagerup subfactor \cite{MR1686551} to a $3^{\bbZ/4\bbZ\times \bbZ/2\bbZ}$ subfactor via Morita equivalence and equivariantization. 
(Some preliminary details can be found in \cite{1202.4396}.)
The 2D2 and 4442 subfactor planar algebras are also related to Izumi's $3^{\bbZ/4\bbZ}$ and $3^{\bbZ/2\bbZ\times \bbZ/2\bbZ}$ subfactors via equivariantization.
In particular, we have the following corollaries of our main theorems.
 
\begin{cor}\label{cor:Equi}
The unique 2D2 subfactor planar algebra from Theorem \ref{thm:2D2} has a planar automorphism by mapping the uncappable rotational eigenvector $T$ at depth $3$ to its negative.
The fixed point subfactor planar algebra under this automorphism is Izumi's $3^{\bbZ/4\bbZ}$ subfactor planar algebra constructed in \cite{IzumiUnpublished,1308.5197}.
\end{cor}

We prove the above corollary in Section \ref{sec:PrincipalGraphs}, but note that it also follows from \cite{IzumiUnpublished} once we know 2D2 is unique.
The following corollary is due to Izumi using his Cuntz algebra model \cite{IzumiUnpublished}; we briefly discuss it in Section \ref{sec:4442}. 

\begin{cor}\label{cor:EquiZ2Z2}
By \cite{MR3157990,1208.3637}, Izumi's $3^{\bbZ/2\bbZ\times \bbZ/2\bbZ}$ subfactor planar algebra is generated by the 3 minimal projections at depth 4.
There is a planar automorphism which cyclically permutes the minimal projections, and the fixed point planar algebra is the 4442 subfactor planar algebra.
\end{cor}

\paragraph{Acknowledgements.}
The author would like to thank Pinhas Grossman, Masaki Izumi, Zhengwei Liu, and Emily Peters for helpful conversations.

Kim Morrison would like to thank the Erwin Schr\"odinger Institute and its 2014 programme on ``Modern Trends in Topological Quantum Field Theory'' for their hospitality.
Both authors would like to thank the Banff International Research Station for hosting the 2014 workshop on Subfactors and Fusion Categories.

Kim Morrison was supported by an Australian Research Council Discovery Early Career Researcher Award, DE120100232 and Discovery Project `Subfactors and symmetries' DP140100732.
David Penneys was supported in part by the Natural Sciences and Engineering Research Council of Canada.
Both authors were supported in part by DOD-DARPA grant HR0011-12-1-0009.

%%%%%%%%%%%%%%%%%%%%%%%%%%%%%%%%%%%%%%%%%%%%%%%%%%%%%%%%%%%%%%%%%%%%%%%%%%%
\section{Background}

We refer the reader to \cite{MR2972458,MR999799,MR3157990} for the definition of a subfactor planar algebra and its principal graphs.

%%%%%%%%%%%%%%%%%%%%%%%%%%%%%%%%%%%%%%%%%%%%%%%%%%%%%%%%%%%%%%%%%%%%%%%%%%%
\subsection{Graph planar algebra embedding and {\tt FusionAtlas}}

In \cite{MR1865703}, Jones defined the planar algebra of a bipartite graph.
Such a graph planar algebra $\cG_\bullet$ has most of the characteristics of a subfactor planar algebra; it is finite dimensional ($\dim(\cG_{n,\pm})<\infty$ for all $n\geq 0$), unitary, and spherical, but in general it is not evaluable, that is, we may have  such that $\dim(\cG_{0,\pm}) > 1$ .
Since the graph planar algebra is manifestly unitary, any evaluable planar subalgebra $\cP_\bullet\subset \cG_\bullet$is automatically a subfactor planar algebra.

To date, many subfactor planar algebras have been constructed by finding them as evaluable planar subalgebras of graph planar algebras.
This program was initiated by Jones in \cite{MR1929335}, and it has seen many applications, including the exotic extended Haagerup subfactor planar algebra constructed via Bigelow's jellyfish algorithm \cite{MR2979509}.

By \cite{MR2812459,gpa}, every subfactor planar algebra embeds in the graph planar algebra of its principal graph.
The embedding theorem, known to Jones long before its proof appeared, was exactly Jones' motivation for defining the graph planar algebra in the first place.

Constructing a subfactor planar algebra by finding it inside the graph planar algebra of its principal graph does not rely on the embedding theorem, but it does give us the motivation and courage to perform such computations.
However, finding an evaluable planar subalgebra $\cP_\bullet$ of a graph planar algebra $\cG_\bullet(\Gamma)$ does not imply that the principal graph of $\cP_\bullet$ is $\Gamma$. For example, every $\cG_\bullet(\Gamma)$ contains a canonical Temperley-Lieb planar subalgebra.
Further arguments are necessary to ensure the graph is correct.

In this article, we construct a 2D2 subfactor planar algebra and prove its uniqueness by finding it inside a graph planar algebra.
To accomplish the necessary calculations, which consist of multiplying rather large matrices corresponding to elements of the graph planar algebra, along with some basic linear algebra, we use {\tt FusionAtlas}, a package for {\tt Mathematica} and {\tt Scala} developed by Morrison-Penneys-Peters-Snyder-Tener for such computations.
We refer the reader to \cite[Section 1.1]{1208.3637} for obtaining a local copy and to \cite{MR2914056} for a terse tutorial on its use.

The {\tt arXiv} sources of this article contain in the {\tt code} subdirectory two files:
\begin{itemize}
\item {\tt Generator.nb}, which reconstructs the generator of 2D2 from the terse description in Notation \ref{nota:Generator}, checks it is an uncappable rotational eigenvector, and writes it to disk.
This notebook runs independently from {\tt FusionAtlas}.
\item {\tt G2D2.nb}, which performs all the matrix calculations and linear algebra necessary to derive the relations and prove the lemmas in Sections \ref{sec:3boxes}---\ref{sec:6boxes}.
\end{itemize}

Finally, we note that all calculations for this article, with the exception of Lemmas \ref{lem:Eigenvalue} and \ref{lem:ParameterFamilies}, are performed using the spherical convention for the graph planar algebra.
Those two lemmas use the lopsided convention of \cite{1205.2742} (see also \cite[Section 2.3]{1208.3637}).

%%%%%%%%%%%%%%%%%%%%%%%%%%%%%%%%%%%%%%%%%%%%%%%%%%%%%%%%%%%%%%%%%%%%%%%%%%%
\subsection{Principal graph stability and the jellyfish algorithm}\label{sec:StabilityAndJellyfish}

Suppose we have a subfactor planar algebra $\cP_\bullet$ with principal graphs $(\Gamma_+,\Gamma_-)$.
The article \cite{MR3157990} explains the connection between Bigelow's jellyfish algorithm \cite{MR2577673,MR2979509} and Popa's principal graph stability \cite{MR1334479}.

Recall that given a finite subset $\cS$ of a planar algebra, a train on $\cS$ is a planar diagram built from elements of $\cS$ such that every $S\in \cS$ is adjacent to the external boundary, and moreover, the $\star$ on each $S$ is in the same region as the $\star$ on the boundary, e.g.,
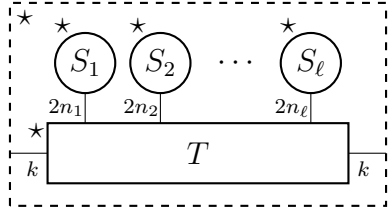
$$
\begin{tikzpicture}[baseline = -.3cm]
	\draw[thick, dashed] (-2, -1.9) rectangle (3,.8);
	\draw (-2,-1.2)--(3,-1.2);
	\draw (-1,0)--(-1,-1);
	\draw (0,0)--(0,-1);
	\draw (2,0)--(2,-1);
	\filldraw[thick, unshaded] (-1.5,-.8) rectangle (2.5,-1.6);
	\draw[thick, unshaded] (-1,0) circle (.4);
	\draw[thick, unshaded] (0,0) circle (.4);	
	\draw[thick, unshaded] (2,0) circle (.4);
	\node at (-1.65,-.9) {$\star$};
	\node at (-1.3,.5) {$\star$};	
	\node at (-.3,.5) {$\star$};
	\node at (1.7,.5) {$\star$};
	\node at (-1.8,.6) {$\star$};
	\node at (1,0) {$\cdots$};
	\node at (.5,-1.2) {$T$};
	\node at (-1,0) {$S_1$};
	\node at (0,0) {$S_2$};
	\node at (2,0) {$S_\ell$};
	\node at (-1.7,-1.4) {{\scriptsize{$k$}}};
	\node at (2.7,-1.4) {{\scriptsize{$k$}}};
	\node at (-1.25,-.6) {{\scriptsize{$2n_1$}}};
	\node at (-.25,-.6) {{\scriptsize{$2n_2$}}};
	\node at (1.75,-.6) {{\scriptsize{$2n_\ell$}}};
\end{tikzpicture}
$$
where $S_1,\dots, S_\ell\in \cS$ and $T$ is a single Temperley-Lieb diagram (we usually suppress the external disk, and the external star goes in the upper left corner).

Recall that a principal graph $\Gamma_\pm$ is stable at depth $n$ if $\Gamma_\pm$ does not merge or split between depths $n$ and $n+1$, and all edges between depths $n$ and $n+1$ are simple.
By \cite{MR1334479,MR3157990}, $\Gamma_\pm$ is stable at depth $n$ if and only if trains from $\cP_{n,\pm}$ span $\cP_{n+1,\pm}$.
By the main theorem of \cite{MR3157990}, if the principal graph $\Gamma_+$ is stable at depths $n$ and $n+1$ and the modulus $\delta>2$, then $\Gamma_+$ is stable for all depths $k\geq n$, and $\Gamma_\pm$ is finite.
This means that trains from $\cP_{n,+}$ span $\cP_{k,+}$ for all $k\geq 0$.
Using this theorem, we see that the jellyfish algorithm is universal for finite depth subfactor planar algebras in the following sense.

\begin{thm}[Universal Jellyfish Algorithm]\label{thm:EvaluableSubalgebra}
Suppose $\cP_\bullet$ is a shaded planar algebra in which closed loops are multiples of the empty diagram.
Suppose $\{\cS_{k,+}\}_{k}$ is a collection of finite subsets $\cS_{k,+} \subset \cP_{k,+}$, and let $\cS_{\leq k}=\bigcup_{j\leq k} \cS_{j,+}$.
Suppose the sets $\cS_{k,+}$ satisfy the following:
\begin{enumerate}[(1)]
\item 
for some $n$, $\cS_{n,+}$ generates $\cP_\bullet$ as  a planar algebra, 
\item
$\cS_{n,+}$ has 2-strand jellyfish relations, i.e., for all $s\in \cS_{n,+}$,
$\begin{tikzpicture}[baseline = -.1cm]
	\fill[shaded] (-.5,-.6)--(-.5,.6) -- (-.65,.6) -- (-.65,-.6);
	\draw (0,-.6)--(0,.6);
	\draw (-.5,-.6)--(-.5,.6);
	\draw (-.65,-.6)--(-.65,.6);
	\node at (.15,.5) {{\scriptsize{$n$}}};
	\node at (.15,-.5) {{\scriptsize{$n$}}};
	\nbox{unshaded}{(0,0)}{.35}{0}{0}{$s$};
\end{tikzpicture}
\in\spann(\trains_{n+2,+}(\cS_{\leq n}))$,
\item 
for $j<n$, adding any cap to any element of $\cS_{j+1,+}$, except on the left, gives an element of $\spann(\cS_{j,+})\oplus \cT\cL_{j,+}$, and
\item
for all $j\leq k\leq n$, applying the tangle 
$
\begin{tikzpicture}[baseline = -.1cm, xscale=-1]
	\draw (0,1.3)--(0,.4);
	\draw (0,-1.3)--(0,-.4);		
	\node at (.15,1.2) {{\scriptsize{$j$}}};
	\node at (.15,0) {{\scriptsize{$j$}}};
	\node at (.15,-1.2) {{\scriptsize{$k$}}};
	\node at (-1.1,.6) {{\scriptsize{$k-j$}}};
	\draw (0,1.3) -- (0,-.35) -- (-.25,-.35) -- (-.25,-.25) arc (0:90:.15cm) arc (-90:-180:.25cm) -- (-.65,1.3);
	\nbox{unshaded}{(0,-.6)}{.35}{0}{0}{}
	\nbox{unshaded}{(0,.6)}{.35}{0}{0}{}
\end{tikzpicture}
$
to an element in $\cS_{k,+}\times \cS_{j,+}$, or the reflection to an element in $\cS_{j,+}\times \cS_{k,+}$, gives an element in $\spann(\cS_{k,+})\oplus\cT\cL_{k,+}$.
\end{enumerate}
Then, for $k \leq n$, 
$
\cP_{k,+} = \trains_{k,+}\left(\cS_{\leq k}\right).
$ 
In particular, if $\cS_{0,+} = \emptyset$, $\cP_\bullet$ is evaluable.
\end{thm}

\begin{remark}
If $\cP_\bullet$ is a planar $*$-algebra, it is not necessary to separately check the reflection of the absorption tangle in item 4).
\end{remark}

We first prove a helpful lemma about reduced trains.
Recall that a train is called {\it reduced} if no generator is connected to itself, and no generator is connected to one of its neighbors by at least half its strands.

\begin{lem}\label{lem:ReducedTrains}
A reduced train on $\cS_{\leq n}$ in $\cP_{k,+}$ with $m\geq 1$ generators (a reduced $m$-car train) is, in fact, a reduced train on $\cS_{\leq k-m+1}$ with no generators from higher $\cS_{j,+}$'s.
\end{lem}
\begin{proof}
We proceed by induction on $m$.
The base case $m=1$ is trivial.
Suppose we have a reduced train on $\cS_{\leq n}$ with $m>1$ generators in $\cP_{k,+}$.
Label the generators across the top by $s_i\in \cS_{j_i,+}$ for $i=1,\dots, m$.

Since the train is reduced, every generator must connect to the external boundary.
To see this, construct a planar polygon with diagonals whose vertices are the train generators and an extra vertex corresponding to the external boundary, with weighted edges counting the relevant numbers of strands. The elementary proof of \cite[Theorem 3.9]{1308.5723} now shows that this planar polygon either has every generator vertex connected to the external vertex, or there is a corner (that is, a vertex with no incident diagonals) which is a generator, and neither the first or last generator in the train. This corner must be connected to a nearest neighbor by at least half its strands, contradicting our assumption that the train is reduced.

Suppose $s_{m-1},s_m$ are connected by $a$ strings as in the following picture:
$$
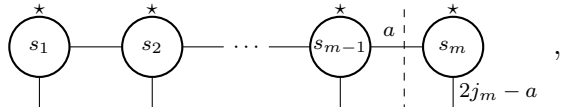
\begin{tikzpicture}[baseline = -.1cm]
	\draw (0,0) -- (2.45,0);
	\draw (3.05,0) -- (5.5,0);
	\draw[dashed] (4.85,.5) -- (4.85,-.8);
	\foreach \x in {0,1.5,4,5.5}
	{
	\draw (\x,0) -- (\x,-.8);
	}
	\ncircle{unshaded}{(0,0)}{.4}{90}{\scriptsize{$s_1$}}
	\ncircle{unshaded}{(1.5,0)}{.4}{90}{\scriptsize{$s_2$}}
	\ncircle{unshaded}{(4,0)}{.4}{90}{\scriptsize{$s_{m-1}$}}
	\ncircle{unshaded}{(5.5,0)}{.4}{90}{\scriptsize{$s_m$}}
	\node at (2.75,0) {\scriptsize{$\cdots$}};
	\node at (4.65,.2) {\scriptsize{$a$}};
	\node at (6.1,-.6) {\scriptsize{$2j_m-a$}};
\end{tikzpicture}\,,
$$
with possibly some caps interspersed underneath which only connect to the external boundary.
Cutting along these $a$ strings, we see the concatenation of a reduced $(m-1)$-car train and a reduced 1-car train.
Since $a\leq j_m-1$, we see $a< 2j_m - a$, so this $(m-1)$-car train has strictly fewer boundary points, i.e., it lives in $\cP_{\ell,+}$ for some $\ell\leq k-1$. 
By the induction hypothesis, for $i=1,\dots, m-1$,
$$
j_i\leq \ell-(m-1)+1\leq (k-1)-(m-1)+1=k-m+1.
$$
Now we repeat the above argument by cutting between the first two cars of the original reduced $m$-car train, which lets us conclude that $j_m \leq k-m+1$ as well.
\end{proof}

\begin{proof}[Proof of Theorem \ref{thm:EvaluableSubalgebra}]
The proof is a generalization of the jellyfish algorithm.
We need to write an element $x\in \cP_{k,+}$ as a linear combination of trains on $\cS_{\leq k}$.
We use the following steps:
\begin{enumerate}[(1)]
\item
Since $\cS_{n,+}$ generates $\cP_\bullet$ and has 2-strand jellyfish relations, for every $j\geq 0$, trains on $\cS_{\leq n}$ span $\cP_{j,+}$ by \cite[Lemma 3.7]{MR3157990}.
We can thus write $x$ as a linear combination of trains on $\cS_{\leq n}$.
This is analogous to the first step of the usual jellyfish algorithm.

\item
Using properties (3) and (4), we may reduce $x\in \cP_{k,+}$ to a linear combination of reduced trains on $\cS_{\leq n}$.
This is analogous to the second step of the usual jellyfish algorithm.
By Lemma \ref{lem:ReducedTrains}, any reduced train on $\cS_{\leq n}$ in $\cP_{k,+}$ is actually a reduced train on $\cS_{\leq k}$.
\qedhere
\end{enumerate}
\end{proof}

\begin{remark}\label{rem:FiniteDepth}
Given a subfactor planar algebra $\cP_\bullet$ of finite depth $n$, one can use the previous algorithm to give a universal jellyfish presentation as follows.
For every $j\leq n$, let $\cS_{j,+}$ be a $*$-closed basis of $\cP_{j,+}\cap \cT\cL_{j,+}^\perp$.
With these large generating sets, properties (3) and (4) of Theorem \ref{thm:EvaluableSubalgebra} necessarily hold, but it may be difficult to determine the structure constants.
Since $\Gamma_+$ is stable at depths $n$ and $n+1$, we know that trains from $\cS_{n,+}$ span $\cP_{k,+}$ for all $k\geq 0$.
\end{remark}

In practice this approach is too inefficient, and it is difficult to work with such a presentation. The essential idea of this article is that one can often, with care, choose smaller subsets for the $\cS_{j,+}$'s, for which it is still possible to verify the hypotheses of the above theorem. As an example,  if $\Gamma_+$ stabilizes before the final depth, we can always choose the higher $\cS_{j,+}$'s to be empty.

The next example motivates the calculations performed in Sections \ref{sec:3boxes}---\ref{sec:6boxes} and the results proved in Section \ref{sec:Evaluation}.

\begin{ex}
Suppose $\cP_\bullet$ is a subfactor planar algebra with principal graph 2D2.
Since 2D2 is 2-supertransitive, we choose $\cS_{j,+}=\emptyset$ for $j\leq 2$.
We let $\cS_{3,+}=\{T\}$, where $T$ is the new uncappable rotational eigenvector, which must satisfy $T^2=\jw{3}$.
Finally, we let $\cS_{4,+}$ be the system of matrix units corresponding to the vertex at depth 4, where the minimal projections are the projections obtained from the projections $(\jw{3}\pm T)/2$ via the generalized Wenzl relation \cite[Lemma 3.7]{MR2922607}, \cite[Lemma 2.11]{1307.5890}.

By construction, it is easy to derive all the necessary structure constants of the annular and absorption maps amongst the $\cS_{j,+}$'s.
The universal jellyfish algorithm now shows that for each set of 2-strand jellyfish relations for $\cS_{4,+}$, the corresponding finitely presented planar algebra is evaluable (but possibly the zero planar algebra).
\end{ex}

As we will see, for any element $T$ satisfying $T^2=\jw{3}$ in the graph planar algebra of 2D2, we can derive 2-strand jellyfish relations for the corresponding elements of $\cS_{4,+}$ by Theorem \ref{thm:2Strand}.
By Lemma \ref{lem:ParameterFamilies} there is only a 1-parameter family (up to sign) of such elements $T$, yet the resulting finite presentations are independent of the parameter by Theorem \ref{thm:Evaluable}, establishing the uniqueness results in Theorems \ref{thm:2D2} and \ref{thm:Unique}.

 %%%%%%%%%%%%%%%%%%%%%%%%%%%%%%%%%%%%%%%%%%%%%%%%%%%%%%%%%%%%%%%%%%%%%%%%%%%
\section{The 3-box spaces}\label{sec:3boxes}

Let $\cG_\bullet$ denote the graph planar algebra \cite{MR1865703} of 2D2.

\begin{remark}\label{rem:ProofsOmitted}
Every lemma in Sections \ref{sec:3boxes}---\ref{sec:6boxes} is a direct, but tedious calculation in $\cG_\bullet$.
These calculations can be found in the {\tt Mathematica} notebook {\tt G2D2} bundled with the {\tt arXiv} sources of this article.
Moreover, after Lemma \ref{lem:ParameterFamilies}, each subsequent lemma holds \emph{independent} of the parameter $\lambda$ introduced in Lemma \ref{lem:ParameterFamilies} below.
\end{remark}

\begin{lem}\label{lem:Eigenvalue}
Any solution to $T^2=\jw{3}$ in $\cG_{3,+}$ which is an uncappable rotational eigenvector must have rotational eigenvalue $\omega_T=1$.
\end{lem}

\begin{remark}
We know that a low weight rotational eigenvector $T$ generating a subfactor planar algebra with principal graphs as in Theorem \ref{thm:2D2} must have $\omega_T=1$ by \cite[Proposition 3.15]{1307.5890}.
\end{remark}

\begin{lem}\label{lem:ParameterFamilies}
There are exactly two 1-parameter families $T=\pm T(\lambda)\in\cG_{3,+}$ for $\lambda\in \bbT$ such that:
\begin{itemize}
\item $T$ is an uncappable rotational eigenvector with eigenvalue $\omega_T=1$,
\item $T$ is self-adjoint, and
\item $T^2=\jw{3}\in \cG_{3,+}$.
\end{itemize}
These solutions also satisfy $\cF(T)^2 = \jw{3}\in \cG_{3,-}$.
\end{lem}

\begin{nota}\label{nota:Generator}
We give the one parameter family $T=+T(\lambda)$ for $\lambda\in \bbT$.
We may specify $T$ by its values on loops of length 6 on 2D2 which stay within the central diamond by \cite[Appendix A]{1308.5197}, where the vertices are labelled $W,S,E,N$ counter-clockwise starting at the left:
$$
2D2=
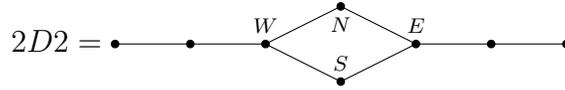
\begin{tikzpicture}[baseline=-.1cm]
	\filldraw (0,0) circle (.05cm);
	\filldraw (1,0) circle (.05cm);
	\filldraw (2,0) circle (.05cm) node [above] {\scriptsize{$W$}};
	\filldraw (3,.5) circle (.05cm) node [below] {\scriptsize{$N$}};
	\filldraw (3,-.5) circle (.05cm) node [above] {\scriptsize{$S$}};
	\filldraw (4,0) circle (.05cm) node [above] {\scriptsize{$E$}};
	\filldraw (5,0) circle (.05cm);
	\filldraw (6,0) circle (.05cm);
	\draw (0,0)--(2,0)--(3,.5)--(4,0)--(6,0);
	\draw (2,0)--(3,-.5)--(4,0);
\end{tikzpicture}
$$
Additional calculations are necessary to show that the expanded generator is, in fact, a self-adjoint uncappable rotational eigenvector.
These computations are carried out in the {\tt Mathematica} notebook {\tt Genetator.nb}, which runs independently from {\tt FusionAtlas}.

\begin{align*}
T(\mbox{WSWSWS}) & = 2-\sqrt{5} &
T(\mbox{WSWSWN}) & = \frac{1}{2} \left(3-\sqrt{5}\right) \displaybreak[1]\\
T(\mbox{WSWSES}) & = \sqrt{5}-2 &
T(\mbox{WSWSEN}) & = -\frac{\sqrt{\frac{1}{2} \left(5 \sqrt{5}-11\right)}}{\lambda } \displaybreak[1]\\
T(\mbox{WSWNWN}) & = \frac{1}{2} \left(\sqrt{5}-3\right) &
T(\mbox{WSWNES}) & = -\sqrt{\frac{1}{2} \left(5 \sqrt{5}-11\right)} \lambda  \displaybreak[1]\\
T(\mbox{WSWNEN}) & = \frac{1}{2} \left(3-\sqrt{5}\right) &
T(\mbox{WSESWN}) & = \frac{1}{2} \left(\sqrt{5}-3\right) \displaybreak[1]\\
T(\mbox{WSESES}) & = 2-\sqrt{5} &
T(\mbox{WSESEN}) & = \frac{\sqrt{\frac{1}{2} \left(5 \sqrt{5}-11\right)}}{\lambda } \displaybreak[1]\\
T(\mbox{WSENWN}) & = -\frac{\sqrt{\frac{1}{2} \left(5 \sqrt{5}-11\right)}}{\lambda } &
T(\mbox{WSENES}) & = \frac{1}{2} \left(3-\sqrt{5}\right) \displaybreak[1]\\
T(\mbox{WSENEN}) & = \frac{\sqrt{\frac{1}{2} \left(5 \sqrt{5}-11\right)}}{\lambda } &
T(\mbox{WNWNWN}) & = \sqrt{5}-2 \displaybreak[1]\\
T(\mbox{WNWNES}) & = -\sqrt{\frac{1}{2} \left(5 \sqrt{5}-11\right)} \lambda  &
T(\mbox{WNWNEN}) & = 2-\sqrt{5} \displaybreak[1]\\
T(\mbox{WNESES}) & = \sqrt{\frac{1}{2} \left(5 \sqrt{5}-11\right)} \lambda  &
T(\mbox{WNESEN}) & = \frac{1}{2} \left(\sqrt{5}-3\right) \displaybreak[1]\\
T(\mbox{WNENES}) & = \sqrt{\frac{1}{2} \left(5 \sqrt{5}-11\right)} \lambda  &
T(\mbox{WNENEN}) & = \sqrt{5}-2 \displaybreak[1]\\
T(\mbox{ESESES}) & = \sqrt{5}-2 &
T(\mbox{ESESEN}) & = \frac{1}{2} \left(\sqrt{5}-3\right) \displaybreak[1]\\
T(\mbox{ESENEN}) & = \frac{1}{2} \left(3-\sqrt{5}\right) &
T(\mbox{ENENEN}) & = 2-\sqrt{5}
\end{align*}
\end{nota}

\begin{defn}\label{defn:3boxes}
Define the following elements of $\cG_\bullet$ in terms of $T=+T(\lambda)$:
$$
e_{3,+}^1 = \frac{1}{2}(\jw{3}+T)
\text{ and }
e_{3,+}^2 = \frac{1}{2}(\jw{3}-T).
$$
Note that these definitions immediately imply
\begin{itemize}
\item the $e_{3,+}^i$ are rectangularly uncappable projections, and
\item the partial trace 
$
\begin{tikzpicture}[baseline = -.1cm]
	\draw (0,-.6)--(0,.6);
	\draw (.25,.35) arc (180:0:.1cm) -- (.45,-.35) arc (0:-180:.1cm);
	\node at (-.15,.5) {{\scriptsize{$2$}}};
	\node at (-.15,-.5) {{\scriptsize{$2$}}};
	\nbox{unshaded}{(0,0)}{.35}{0}{0}{$e_{3,+}^i$};
\end{tikzpicture}
=
\displaystyle
\frac{\sqrt{7+3\sqrt{5}}}{2+\sqrt{5}}\jw{2}$,
so $\Tr(e_{3,+}^i)=\sqrt{7+3\sqrt{5}}$.
\end{itemize}
Let $\cP_\bullet=\cP\cA(T)$ be the planar algebra generated by $T$ inside $\mathcal{G}_\bullet$.
\end{defn}

\begin{remark}
Taking $T=-T(\lambda)$ switches $e_{3,+}^1$ with $e_{3,+}^2$.
\end{remark}

Let $e_{3,-}^i=\cF^3(e_{3,+}^i)$ for $i=1,2$.
Recall that $\cF(T)$ must be a multiple of $\check{T}=e_{3,-}^1-e_{3,-}^2$.

\begin{lem}
In fact, $\check{T}=\cF(T)$.
\end{lem}

\begin{remark}
Any subfactor with principal graphs as in Theorem \ref{thm:2D2} must have $\check{T}=\cF(T)$ by \cite[Proposition 3.15]{1307.5890}.
\end{remark}

%%%%%%%%%%%%%%%%%%%%%%%%%%%%%%%%%%%%%%%%%%%%%%%%%%%%%%%%%%%%%%%%%%%%%%%%%%%
\section{The 4-box spaces}\label{sec:4boxes}

\begin{nota}
For $X,Y\in \cP_{n,\pm}$, let
$
X\circ Y = 
\begin{tikzpicture}[baseline = -.1cm]
	\draw (0,-.9)--(0,.9);
	\node at (-.15,.8) {{\scriptsize{$n$}}};
	\node at (-.15,-.8) {{\scriptsize{$n$}}};
	\node at (-.35,0) {{\scriptsize{$n-1$}}};
	\draw (.15,.15) arc (-180:0:.1cm) -- (.35,.9);
	\draw (.15,-.15) arc (180:0:.1cm) -- (.35,-.9);
	\nbox{unshaded}{(0,.4)}{.25}{0}{0}{$Y$};
	\nbox{unshaded}{(0,-.4)}{.25}{0}{0}{$X$};
\end{tikzpicture}
$\, and 
$X\star Y = \cF(\cF(X)\circ \cF(Y))$.
\end{nota}

\begin{defn}\label{defn:4boxes}
For $i=1,2$, define the following elements of $\cG_\bullet$ in terms of $T=+T(\lambda)$:
\begin{align*}
e_{4,+}^{i,i} &= 
\begin{tikzpicture}[baseline = -.1cm]
	\draw (0,-.6)--(0,.6);
	\draw (.5,-.6)--(.5,.6);
	\node at (-.15,.5) {{\scriptsize{$3$}}};
	\node at (-.15,-.5) {{\scriptsize{$3$}}};
	\nbox{unshaded}{(0,0)}{.35}{0}{0}{$e_{3,+}^{i}$};
\end{tikzpicture}
-
\left(\frac{2+\sqrt{5}}{\sqrt{7+3\sqrt{5}}}\right)
e_{3,+}^i\circ e_{3,+}^i.
\end{align*}
Note that these definitions immediately imply
\begin{itemize}
\item the $e^{i,i}_{4,+}$ are rectangularly uncappable projections, and
\item the partial trace 
$
\begin{tikzpicture}[baseline = -.1cm]
	\draw (0,-.6)--(0,.6);
	\draw (.25,.35) arc (180:0:.1cm) -- (.45,-.35) arc (0:-180:.1cm);
	\node at (-.15,.5) {{\scriptsize{$3$}}};
	\node at (-.15,-.5) {{\scriptsize{$3$}}};
	\nbox{unshaded}{(0,0)}{.35}{0}{0}{$e_{4,+}^{i,i}$};
\end{tikzpicture}
=
\displaystyle
\left(\frac{2+\sqrt{5}}{\sqrt{7+3\sqrt{5}}}\right) 
e_{3,+}^i$,
so $\Tr(e_{4,+}^{i,i})=2+\sqrt{5}$.
\end{itemize}
Since $2D2$ has annular multiplicities $*12$, we will find two uncappable rotational eigenvectors at depth $4$.
Now since $T^2=\jw{3}$ and $\check{T}^2=\check{f}^{(3)}$, we have that $\Tr(T^3)=\Tr(\check{T}^3)=0$, and thus $T\circ T$ and $T\star T$ are both perpendicular to the annular consequences of $T$.
Hence we use linear combinations of $T\circ T, T\star T,\jw{4},\cF(\check{f}^{(4)})$ to define
\begin{align*}
B &= 
\left(\frac{3}{4} \sqrt{3 + \sqrt{5}}\right)
T\circ T
-
\jw{4}
\\
A &=
i\left(
\left(\frac{\sqrt{2 + \sqrt{5}}}{2}\right) 
T\star T
-
\left(\frac{\sqrt{1 + \sqrt{5}}}{3}\right) 
\cF(\check{f}^{(4)})
-
\left(\frac{1}{3} \sqrt{-2 + \sqrt{5}}\right)
B
\right).
\end{align*}
We then use $A$ to define $e_{4,+}^{1,2} = e_{4,+}^{1,1} A e_{4,+}^{2,2}$ and $e_{4,+}^{2,1}=(e_{4,+}^{1,2})^*=e_{4,+}^{2,2}A^*e_{4,+}^{1,1}$.
Note that $e_{4,+}^{1,2},e_{4,+}^{2,1}$ are uncappable by definition.
\end{defn}

\begin{lem}\label{lem:3ZMod4Generators}
The elements $A$ and $B$ are self-adjoint uncappable rotational eigenvectors which have rotational eigenvalues $\omega_{A}=-1$ and $\omega_{B}=1$.
\end{lem}

\begin{lem}\label{lem:3ZMod4}
The elements $A$ and $B$ have the moments and tetrahedral structure constants as listed in \cite[Appendix B.3]{1308.5197}, and $\{A,B,f^{(4)}\}$ and $\{\check{A},\check{B},\check{f}^{(4)}\}$ span algebras under the usual multiplication.
Hence by \cite{1308.5197}, $A$ and $B$ generate Izumi's $3^{\bbZ/4\bbZ}$ subfactor planar algebra.
\end{lem}

\begin{lem}\label{lem:S4Plus}
The set $\cS_{4,+}=\{e_{4,+}^{i,j}\}$ forms a system of matrix units for a copy of $M_2(\bbC)$.
\end{lem}

\begin{remark}
Taking $T=-T(\lambda)$ also switches $e_{4,+}^{1,1}$ with $e_{4,+}^{2,2}$ and $e_{4,+}^{1,2}$ with $e_{4,+}^{2,1}$ since $A$ is self-adjoint by Lemma \ref{lem:3ZMod4Generators}.
\end{remark}

%%%%%%%%%%%%%%%%%%%%%%%%%%%%%%%%%%%%%%%%%%%%%%%%%%%%%
\subsection{The dual graph}

\begin{defn}\label{defn:gi}
For $i=1,2$, define the following elements of $\cG_{4,-}$:
\begin{align*}
g_i &= 
\begin{tikzpicture}[baseline = -.1cm]
	\draw (0,-.6)--(0,.6);
	\draw (.5,-.6)--(.5,.6);
	\node at (-.15,.5) {{\scriptsize{$3$}}};
	\node at (-.15,-.5) {{\scriptsize{$3$}}};
	\nbox{unshaded}{(0,0)}{.35}{0}{0}{$e_{3,-}^{i}$};
\end{tikzpicture}
-
\left(\frac{2+\sqrt{5}}{\sqrt{7+3\sqrt{5}}}\right)
e_{3,-}^i\circ e_{3,-}^i.
\end{align*}
Again, it is clear that
\begin{itemize}
\item the $g_i$ are rectangularly uncappable projections,
\item the partial trace 
$
\begin{tikzpicture}[baseline = -.1cm]
	\draw (0,-.6)--(0,.6);
	\draw (.25,.35) arc (180:0:.1cm) -- (.45,-.35) arc (0:-180:.1cm);
	\node at (-.15,.5) {{\scriptsize{$3$}}};
	\node at (-.15,-.5) {{\scriptsize{$3$}}};
	\nbox{unshaded}{(0,0)}{.35}{0}{0}{$g_i$};
\end{tikzpicture}
=
\displaystyle
\left(\frac{2+\sqrt{5}}{\sqrt{7+3\sqrt{5}}}\right) e_{3,-}^{i}$,
so $\Tr(g_i)=2+\sqrt{5}$.
\end{itemize}
Now we define the following elements of $\cG_{4,-}$:
\begin{align*}
e_{4,-}^1=\frac{1}{2} (1-\sqrt{5}) g_1- \left(\sqrt{1+\sqrt{5}}\right)g_1 \check{A} g_1 
\\
e_{4,-}^2=\frac{1}{2} (1+\sqrt{5}) g_1 + \left(\sqrt{1+\sqrt{5}}\right)g_1 \check{A} g_1 
\\
e_{4,-}^3=\frac{1}{2} (1+\sqrt{5}) g_2 + \left(\sqrt{1+\sqrt{5}}\right)g_2 \check{A} g_2 
\\
e_{4,-}^4=\frac{1}{2} (1-\sqrt{5}) g_2 - \left(\sqrt{1-\sqrt{5}}\right)g_2 \check{A} g_2 
\end{align*}
where $\check{A}:=-i\cF(A)$ (as in \cite{1308.5197}, due to Lemma \ref{lem:3ZMod4}).
Note $e_{4,-}^1+e_{4,-}^2=g_1$ and $e_{4,-}^3+e_{4,-}^4=g_2$.
\end{defn}

\begin{lem}\label{lem:S4Minus}
Let $\cS_{4,-}=\{e_{4,-}^i\}$.
We have:
\begin{itemize}
\item 
The elements of $\cS_{4,-}$ are rectangularly uncappable projections.
\item 
$\Tr(e_{4,-}^1)=\Tr(e_{4,-}^4)=\frac{1}{2}(1+\sqrt{5})$ and $\Tr(e_{4,-}^2)=\Tr(e_{4,-}^3)=\frac{1}{2}(3+\sqrt{5})$.
\item
$\cF^4(e_{4,-}^1)=e_{4,-}^1$, $\cF^4(e_{4,-}^2)=e_{4,-}^3$, and $\cF^4(e_{4,-}^4)=e_{4,-}^4$.
\end{itemize}
\end{lem}

%%%%%%%%%%%%%%%%%%%%%%%%%%%%%%%%%%%%%%%%%%%%%%%%%%%%%
\subsection{Rotations}

We now calculate $\cF^4$ on $\cP_{4,+}$ and $\cF:\cP_{4,\pm}\to \cP_{4,\mp}$.
It turns out that there is a nice formula for $\cF^4$ on $\cP_{4,+}$ that is not immediately obvious from the formulas for $\cF: \cP_{4,\pm}\to \cP_{4,\mp}$.

\begin{lem}\label{lem:4ClickRotation}
Let 
$\alpha = \frac{1}{2}(3-\sqrt{5})$,
$\beta = \frac{1}{2}(-1+\sqrt{5})$, and
$\eta = i \sqrt{-2 + \sqrt{5}}$.
The rotation by $\pi$ on $\cP_{4,+}$, $\cF^4$, is given by
$$
\begin{pmatrix}
\overline{e_{4,+}^{1,1}}\\
\overline{e_{4,+}^{1,2}}\\
\overline{e_{4,+}^{2,1}}\\
\overline{e_{4,+}^{2,2}}
\end{pmatrix}
=
\cF^4
\begin{pmatrix}
e_{4,+}^{1,1}\\
e_{4,+}^{1,2}\\
e_{4,+}^{2,1}\\
e_{4,+}^{2,2}
\end{pmatrix}
=
\begin{pmatrix}
\alpha & \overline{\eta} & \eta & \beta\\
\eta & \beta & \alpha &\overline{\eta}\\
\overline{\eta} & \alpha & \beta & \eta\\
\beta & \eta & \overline{\eta} & \alpha
\end{pmatrix}
\begin{pmatrix}
e_{4,+}^{1,1}\\
e_{4,+}^{1,2}\\
e_{4,+}^{2,1}\\
e_{4,+}^{2,2}
\end{pmatrix}.
$$
\end{lem}

\begin{defn}\label{defn:RotationBases}
Recall from \cite{MR2972458} that the annular basis of $\gA_{4,+}(T)$ of annular consequences of $T$ in $\cP_{4,+}$ is given by
$$
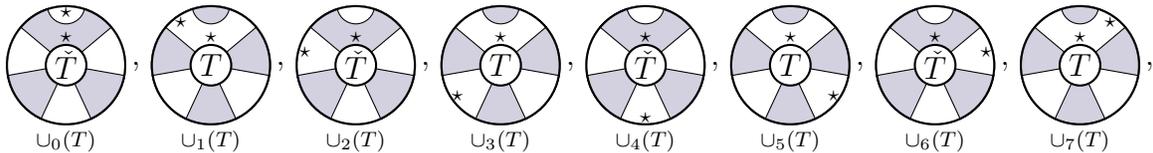
\underset{\cup_0(T)}{
\begin{tikzpicture}[baseline = 0cm]
	\clip (0,0) circle (.8cm);
	\draw[shaded] (0,0) circle (.8cm);
	\draw[unshaded] (90:.8) circle (.25cm);
	\draw[unshaded] (0,0)--(-10:.9cm) arc (-10:40:.9cm) --(0,0); 
	\draw[unshaded] (0,0)--(140:.9cm) arc (140:190:.9cm) --(0,0); 
	\draw[unshaded] (0,0)--(245:.9cm) arc (245:295:.9cm) --(0,0); 
	\draw[thick, unshaded] (0,0) circle (.27cm);
	\node at (90:.37) {\scriptsize{$\star$}};
	\node at (90:.70) {\scriptsize{$\star$}};
	\node at (0,0) {$\check{T}$};
	\draw[ultra thick] (0,0) circle (.8cm);
\end{tikzpicture}
}
\,,\,
\underset{\cup_1(T)}{
\begin{tikzpicture}[baseline = 0cm]
	\clip (0,0) circle (.8cm);
	\draw[unshaded] (0,0) circle (.8cm);
	\draw[shaded] (90:.8) circle (.25cm);
	\draw[shaded] (0,0)--(-10:.9cm) arc (-10:40:.9cm) --(0,0); 
	\draw[shaded] (0,0)--(140:.9cm) arc (140:190:.9cm) --(0,0); 
	\draw[shaded] (0,0)--(245:.9cm) arc (245:295:.9cm) --(0,0); 
	\draw[thick, unshaded] (0,0) circle (.27cm);
	\node at (90:.37) {\scriptsize{$\star$}};
	\node at (125:.70) {\scriptsize{$\star$}};
	\node at (0,0) {$T$};
	\draw[ultra thick] (0,0) circle (.8cm);
\end{tikzpicture}
}
\,,\,
\underset{\cup_2(T)}{
\begin{tikzpicture}[baseline = 0cm]
	\clip (0,0) circle (.8cm);
	\draw[shaded] (0,0) circle (.8cm);
	\draw[unshaded] (90:.8) circle (.25cm);
	\draw[unshaded] (0,0)--(-10:.9cm) arc (-10:40:.9cm) --(0,0); 
	\draw[unshaded] (0,0)--(140:.9cm) arc (140:190:.9cm) --(0,0); 
	\draw[unshaded] (0,0)--(245:.9cm) arc (245:295:.9cm) --(0,0); 
	\draw[thick, unshaded] (0,0) circle (.27cm);
	\node at (90:.37) {\scriptsize{$\star$}};
	\node at (165:.70) {\scriptsize{$\star$}};
	\node at (0,0) {$\check{T}$};
	\draw[ultra thick] (0,0) circle (.8cm);
\end{tikzpicture}
}
\,,\,
\underset{\cup_3(T)}{
\begin{tikzpicture}[baseline = 0cm]
	\clip (0,0) circle (.8cm);
	\draw[unshaded] (0,0) circle (.8cm);
	\draw[shaded] (90:.8) circle (.25cm);
	\draw[shaded] (0,0)--(-10:.9cm) arc (-10:40:.9cm) --(0,0); 
	\draw[shaded] (0,0)--(140:.9cm) arc (140:190:.9cm) --(0,0); 
	\draw[shaded] (0,0)--(245:.9cm) arc (245:295:.9cm) --(0,0); 
	\draw[thick, unshaded] (0,0) circle (.27cm);
	\node at (90:.37) {\scriptsize{$\star$}};
	\node at (215:.70) {\scriptsize{$\star$}};
	\node at (0,0) {$T$};
	\draw[ultra thick] (0,0) circle (.8cm);
\end{tikzpicture}
}
\,,\,
\underset{\cup_4(T)}{
\begin{tikzpicture}[baseline = 0cm]
	\clip (0,0) circle (.8cm);
	\draw[shaded] (0,0) circle (.8cm);
	\draw[unshaded] (90:.8) circle (.25cm);
	\draw[unshaded] (0,0)--(-10:.9cm) arc (-10:40:.9cm) --(0,0); 
	\draw[unshaded] (0,0)--(140:.9cm) arc (140:190:.9cm) --(0,0); 
	\draw[unshaded] (0,0)--(245:.9cm) arc (245:295:.9cm) --(0,0); 
	\draw[thick, unshaded] (0,0) circle (.27cm);
	\node at (90:.37) {\scriptsize{$\star$}};
	\node at (270:.70) {\scriptsize{$\star$}};
	\node at (0,0) {$\check{T}$};
	\draw[ultra thick] (0,0) circle (.8cm);
\end{tikzpicture}
}
\,,\,
\underset{\cup_5(T)}{
\begin{tikzpicture}[baseline = 0cm]
	\clip (0,0) circle (.8cm);
	\draw[unshaded] (0,0) circle (.8cm);
	\draw[shaded] (90:.8) circle (.25cm);
	\draw[shaded] (0,0)--(-10:.9cm) arc (-10:40:.9cm) --(0,0); 
	\draw[shaded] (0,0)--(140:.9cm) arc (140:190:.9cm) --(0,0); 
	\draw[shaded] (0,0)--(245:.9cm) arc (245:295:.9cm) --(0,0); 
	\draw[thick, unshaded] (0,0) circle (.27cm);
	\node at (90:.37) {\scriptsize{$\star$}};
	\node at (325:.70) {\scriptsize{$\star$}};
	\node at (0,0) {$T$};
	\draw[ultra thick] (0,0) circle (.8cm);
\end{tikzpicture}
}
\,,\,
\underset{\cup_6(T)}{
\begin{tikzpicture}[baseline = 0cm]
	\clip (0,0) circle (.8cm);
	\draw[shaded] (0,0) circle (.8cm);
	\draw[unshaded] (90:.8) circle (.25cm);
	\draw[unshaded] (0,0)--(-10:.9cm) arc (-10:40:.9cm) --(0,0); 
	\draw[unshaded] (0,0)--(140:.9cm) arc (140:190:.9cm) --(0,0); 
	\draw[unshaded] (0,0)--(245:.9cm) arc (245:295:.9cm) --(0,0); 
	\draw[thick, unshaded] (0,0) circle (.27cm);
	\node at (90:.37) {\scriptsize{$\star$}};
	\node at (15:.70) {\scriptsize{$\star$}};
	\node at (0,0) {$\check{T}$};
	\draw[ultra thick] (0,0) circle (.8cm);
\end{tikzpicture}
}
\,,\,
\underset{\cup_7(T)}{
\begin{tikzpicture}[baseline = 0cm]
	\clip (0,0) circle (.8cm);
	\draw[unshaded] (0,0) circle (.8cm);
	\draw[shaded] (90:.8) circle (.25cm);
	\draw[shaded] (0,0)--(-10:.9cm) arc (-10:40:.9cm) --(0,0); 
	\draw[shaded] (0,0)--(140:.9cm) arc (140:190:.9cm) --(0,0); 
	\draw[shaded] (0,0)--(245:.9cm) arc (245:295:.9cm) --(0,0); 
	\draw[thick, unshaded] (0,0) circle (.27cm);
	\node at (90:.37) {\scriptsize{$\star$}};
	\node at (55:.70) {\scriptsize{$\star$}};
	\node at (0,0) {$T$};
	\draw[ultra thick] (0,0) circle (.8cm);
\end{tikzpicture}
}
\,,
$$
and similarly for the dual annular consequences $\gA_{4,-}(T)$ spanned by the $\cup_i(\check{T})$'s.

We defined the ordered sets
\begin{align*}
\cB_{4,+} &= \{A,B,\cF(\check{f}^{(4)})\}\cup \set{\cup_i(T)}{0\leq i\leq 7}\\
\cB_{4,-} &= \{\check{A},\check{B},\cF(\jw{4})\}\cup \set{\cup_i(\check{T})}{0\leq i\leq 7}.
\end{align*}
\end{defn}

\begin{lem}\label{lem:1ClickOnS4}
The one click rotation $\cF$ applied to $\cS_{4,\pm}$ is given by the linear combinations of elements of $\cB_{4,\mp}$ given in Figure \ref{fig:1ClickRotationMatrices} in Appendix \ref{sec:1ClickRotation}.
\end{lem}

\begin{lem}\label{lem:Cup0InTrains}
The elements $\cup_0(T)$ and $\cup_0(\check{T})$ are the linear combinations of elements of $\cS_{4,+}\cup \set{\cup_i(T)}{1\leq i\leq 7}$ and $\cS_{4,-}\cup \set{\cup_i(\check{T})}{1\leq i\leq 7}$ respectively given in Figure \ref{fig:Cup0InTrains} in Appendix \ref{sec:1ClickRotation}.
\end{lem}

%%%%%%%%%%%%%%%%%%%%%%%%%%%%%%%%%%%%%%%%%%%%%%%%%%%%%%%%%%%%%%%%%%%%%%%%%%%
\section{The 5-box spaces}\label{sec:5boxes}

\begin{defn}\label{defn:5boxes}
Define the following elements of $\cP_{5,+}\subset \cG_{5,+}$:
\begin{align*}
e_{5,+}^{1,1} &= 
\begin{tikzpicture}[baseline = -.1cm]
	\draw (0,-.6)--(0,.6);
	\draw (.5,-.6)--(.5,.6);
	\node at (-.15,.5) {{\scriptsize{$4$}}};
	\node at (-.15,-.5) {{\scriptsize{$4$}}};
	\nbox{unshaded}{(0,0)}{.35}{0}{0}{$e_{4,+}^{1,1}$};
\end{tikzpicture}
-
\left(\frac{\sqrt{7+3\sqrt{5}}}{2+\sqrt{5}}\right)
\left(
e_{4,+}^{1,1}\circ e_{4,+}^{1,1}
+
e_{4,+}^{1,2}\circ e_{4,+}^{2,1}
\right)
\displaybreak[1]\\
e_{5,+}^{1,2} &= 
\begin{tikzpicture}[baseline = -.1cm]
	\draw (0,-.6)--(0,.6);
	\draw (.5,-.6)--(.5,.6);
	\node at (-.15,.5) {{\scriptsize{$4$}}};
	\node at (-.15,-.5) {{\scriptsize{$4$}}};
	\nbox{unshaded}{(0,0)}{.35}{0}{0}{$e_{4,+}^{1,2}$};
\end{tikzpicture}
-
\left(\frac{\sqrt{7+3\sqrt{5}}}{2+\sqrt{5}}\right)
\left(
(e_{4,+}^{1,1}+e_{4,+}^{1,2})\circ (e_{4,+}^{1,2}+e_{4,+}^{2,2})
\right)
\displaybreak[1]\\
e_{5,+}^{2,2} &= 
\begin{tikzpicture}[baseline = -.1cm]
	\draw (0,-.6)--(0,.6);
	\draw (.5,-.6)--(.5,.6);
	\node at (-.15,.5) {{\scriptsize{$4$}}};
	\node at (-.15,-.5) {{\scriptsize{$4$}}};
	\nbox{unshaded}{(0,0)}{.35}{0}{0}{$e_{4,+}^{2,2}$};
\end{tikzpicture}
-
\left(\frac{\sqrt{7+3\sqrt{5}}}{2+\sqrt{5}}\right)
\left(
e_{4,+}^{2,2}\circ e_{4,+}^{2,2}
+
e_{4,+}^{2,1}\circ e_{4,+}^{1,2}
\right),
\end{align*}
let $e_{5,+}^{2,1}:=(e_{5,+}^{1,2})^*$, and set $\cS_{5,+}=\{e_{5,+}^{i,j}\}$.
Note that these definitions immediately imply
\begin{itemize}
\item 
the $e^{i,j}_{5,+}$ are rectangularly uncappable, and $e_{5,+}^{1,2},e_{5,+}^{1,2}$ are uncappable,
\item 
$\cS_{5,+}$ forms a system of matrix units for a copy of $M_2(\bbC)$, and
\item 
the partial traces of $e_{5,+}^{i,i}$ are given by
$
\begin{tikzpicture}[baseline = -.1cm]
	\draw (0,-.6)--(0,.6);
	\draw (.25,.35) arc (180:0:.1cm) -- (.45,-.35) arc (0:-180:.1cm);
	\node at (-.15,.5) {{\scriptsize{$3$}}};
	\node at (-.15,-.5) {{\scriptsize{$3$}}};
	\nbox{unshaded}{(0,0)}{.35}{0}{0}{$e_{5,+}^{i,i}$};
\end{tikzpicture}
=\displaystyle
\frac{\sqrt{3+\sqrt{5}}}{2+\sqrt{5}} e_{4,+}^{i,i}
$, 
so $\Tr(e_{5,+}^{i,i})=\sqrt{3+\sqrt{5}}$.
\end{itemize}
\end{defn}

%%%%%%%%%%%%%%%%%%%%%%%%%%%%%%%%%%%%%%%%%%%%%%%%%%%%%%%%%%%%%%%%
\subsection{The dual graph and rotation}

\begin{defn}\label{defn:4minus}
For $i=1,2$, define the following elements of $\cP_{5,-}\subset \cG_{5,-}$:
\begin{align*}
e_{5,-}^{i,i} &= 
\begin{tikzpicture}[baseline = -.1cm]
	\draw (0,-.6)--(0,.6);
	\draw (.5,-.6)--(.5,.6);
	\node at (-.15,.5) {{\scriptsize{$4$}}};
	\node at (-.15,-.5) {{\scriptsize{$4$}}};
	\nbox{unshaded}{(0,0)}{.35}{0}{0}{$e_{4,-}^{i+1}$};
\end{tikzpicture}
-
\left(\frac{\sqrt{7+3\sqrt{5}}}{\frac{1}{2}(3+\sqrt{5})}\right)
\left(
e_{4,-}^{i+1}\circ e_{4,-}^{i+1}
\right).
\end{align*}
\end{defn}

\begin{lem}\label{lem:5ClickRotation}
The 5-click rotation $\cF^5: \cP_{5,-}\to \cP_{5,+}$ restricted to $\{e_{5,-}^{1,1},e_{5,-}^{2,2}\}$ satisfies
$$
\cF^5
\begin{pmatrix}
e_{5,-}^{1,1}\\
e_{5,-}^{2,2}
\end{pmatrix}
=
\begin{pmatrix}
\beta & \eta & \overline{\eta} & \alpha\\
\alpha & \overline{\eta} & \eta & \beta
\end{pmatrix}
\begin{pmatrix}
e_{5,+}^{1,1}\\
e_{5,+}^{1,2}\\
e_{5,+}^{2,1}\\
e_{5,+}^{2,2}
\end{pmatrix}
$$
where $\alpha,\beta,\eta$ are as in Lemma \ref{lem:4ClickRotation}.
\end{lem}

\begin{defn}\label{defn:S5Minus}
By Lemma \ref{lem:5ClickRotation}, we define
$$
\begin{pmatrix}
e_{5,-}^{1,2}
\\
e_{5,-}^{2,1}
\end{pmatrix}
=\cF^5\left(
\begin{pmatrix}
\overline{\eta} & \alpha & \beta & \eta
\\
\eta & \beta & \alpha &\overline{\eta}
\end{pmatrix}
\begin{pmatrix}
e_{5,+}^{1,1}\\
e_{5,+}^{1,2}\\
e_{5,+}^{2,1}\\
e_{5,+}^{2,2}
\end{pmatrix}
\right)
$$
and we set $\cS_{5,-}=\{e_{5,-}^{i,j}\}$.
Since $\cS_{5,-}$ is obtained from $\cS_{5,+}$ by applying $\cF^5$ to rotate by $\pi$ and by a unitary (related to that from Lemma \ref{lem:4ClickRotation}), we immediately have:
\begin{itemize}
\item 
the elements of $\cS_{5,-}$ are rectangularly uncappable, and $e_{5,-}^{1,2},e_{5,-}^{1,2}$ are uncappable.
\item 
$\cS_{5,-}$ forms a system of matrix units for a copy of $M_2(\bbC)$, and
\item
the partial traces 
$
\begin{tikzpicture}[baseline = -.1cm]
	\draw (0,-.6)--(0,.6);
	\draw (.25,.35) arc (180:0:.1cm) -- (.45,-.35) arc (0:-180:.1cm);
	\node at (-.15,.5) {{\scriptsize{$3$}}};
	\node at (-.15,-.5) {{\scriptsize{$3$}}};
	\nbox{unshaded}{(0,0)}{.35}{0}{0}{$e_{5,-}^{i,i}$};
\end{tikzpicture}
=\displaystyle
\left(\frac{\sqrt{3+\sqrt{5}}}{\frac{1}{2}(3+\sqrt{5})}\right)
e_{4,-}^{i+1}
$, 
so $\Tr(e_{5,-}^{i,i})=\sqrt{3+\sqrt{5}}$.
\end{itemize}
\end{defn}

%%%%%%%%%%%%%%%%%%%%%%%%%%%%%%%%%%%%%%%%%%%%%%%%%%%%%%%%%%%%%%%%%%%%%%%%%%%
\section{The 6-box spaces}\label{sec:6boxes}

\begin{defn}\label{defn:6boxes}
Define the following elements of $\cP_{6,+}\subset \cG_{6,+}$:
\begin{align*}
e_{6,+}^{i,i} &= 
\begin{tikzpicture}[baseline = -.1cm]
	\draw (0,-.6)--(0,.6);
	\draw (.5,-.6)--(.5,.6);
	\node at (-.15,.5) {{\scriptsize{$4$}}};
	\node at (-.15,-.5) {{\scriptsize{$4$}}};
	\nbox{unshaded}{(0,0)}{.35}{0}{0}{$e_{5,+}^{1,1}$};
\end{tikzpicture}
-
\left(\frac{2+\sqrt{5}}{\sqrt{3+\sqrt{5}}}\right)
\left(
e_{5,+}^{i,i}\circ e_{5,+}^{i,i}
\right)
\text{ for $i=1,2$, and}
\displaybreak[1]\\
e_{6,+}^{1,2} &= 
\begin{tikzpicture}[baseline = -.1cm]
	\draw (0,-.6)--(0,.6);
	\draw (.5,-.6)--(.5,.6);
	\node at (-.15,.5) {{\scriptsize{$4$}}};
	\node at (-.15,-.5) {{\scriptsize{$4$}}};
	\nbox{unshaded}{(0,0)}{.35}{0}{0}{$e_{5,+}^{1,2}$};
\end{tikzpicture}
-
\left(\frac{2+\sqrt{5}}{\sqrt{3+\sqrt{5}}}\right)
\left(
e_{5,+}^{1,1}\circ e_{5,+}^{1,2}
\right).
\end{align*}
let $e_{6,+}^{2,1}=(e_{6,+}^{1,2})^*$, and set $\cS_{6,+}=\{e_{6,+}^{i,j}\}$.
Note that these definitions immediately imply
\begin{itemize}
\item 
the $e^{i,j}_{6,+}$ are rectangularly uncappable, and $e_{6,+}^{1,2},e_{6,+}^{2,1}$ are uncappable,
\item 
$\cS_{6,+}$ forms a system of matrix units for a copy of $M_2(\bbC)$, and
\item 
the partial traces
$
\begin{tikzpicture}[baseline = -.1cm]
	\draw (0,-.6)--(0,.6);
	\draw (.25,.35) arc (180:0:.1cm) -- (.45,-.35) arc (0:-180:.1cm);
	\node at (-.15,.5) {{\scriptsize{$3$}}};
	\node at (-.15,-.5) {{\scriptsize{$3$}}};
	\nbox{unshaded}{(0,0)}{.35}{0}{0}{$e_{6,+}^{i,i}$};
\end{tikzpicture}
=\displaystyle
\left(\frac{1}{\sqrt{3+\sqrt{5}}}\right)
e_{5,+}^{i,i}
$, 
so $\Tr(e_{6,+}^{i,i})=1$.
\end{itemize}
\end{defn}

\begin{lem}\label{lem:6ClickRotation}
The 6-click rotation $\cF^6: \cP_{6,+}\to \cP_{6,+}$ restricted to the $\{e_{6,+}^{i,j}\}$ is given by
$$
\begin{pmatrix}
\overline{e_{6,+}^{1,1}}\\
\overline{e_{6,+}^{1,2}}\\
\overline{e_{6,+}^{2,1}}\\
\overline{e_{6,+}^{2,2}}
\end{pmatrix}
=
\cF^6
\begin{pmatrix}
e_{6,+}^{1,1}\\
e_{6,+}^{1,2}\\
e_{6,+}^{2,1}\\
e_{6,+}^{2,2}
\end{pmatrix}
=
\begin{pmatrix}
0 & 0 & 0 & 1\\
0 & -1 & 0 & 0\\
0 & 0 & -1 & 0\\
1 & 0 & 0 & 0
\end{pmatrix}
\begin{pmatrix}
e_{6,+}^{1,1}\\
e_{6,+}^{1,2}\\
e_{6,+}^{2,1}\\
e_{6,+}^{2,2}
\end{pmatrix}.
$$
\end{lem}
%%%%%%%%%%%%%%%%%%%%%%%%%%%%%%%%%%%%%%%%%%%%%%%%%%%%%%%%%%%%%%%%%%%%%%%%%%%
\section{Evaluation and principal graphs}\label{sec:Evaluation}

\subsection{2-strand jellyfish relations}

We now show the relations given in Sections \ref{sec:3boxes}---\ref{sec:6boxes} are sufficient to prove the existence of 2-strand jellyfish relations for $\cS_{4,+}$.

\begin{lem}\label{lem:BInTrains}
$\cB_{4,\pm}\subset \spann(\trains_{4,\pm}(\cS_{4,\pm}))$, where $\cB_{4,\pm}$ is as in Definition \ref{defn:RotationBases}.
\end{lem}
\begin{proof}
We prove the result for $\cB_{4,+}$, and the result for $\cB_{4,-}$ is similar.
An easy calculation using Definition \ref{defn:4boxes} shows
$$
T=
e_{3,+}^1 - e_{3,+}^2
=
\frac{\sqrt{7+3\sqrt{5}}}{2+\sqrt{5}}
\left(
\begin{tikzpicture}[baseline = -.1cm]
	\draw (0,-.6)--(0,.6);
	\draw (.25,.35) arc (180:0:.1cm) -- (.45,-.35) arc (0:-180:.1cm);
	\node at (-.15,.5) {{\scriptsize{$3$}}};
	\node at (-.15,-.5) {{\scriptsize{$3$}}};
	\nbox{unshaded}{(0,0)}{.35}{0}{0}{$e_{4,+}^{1,1}$};
\end{tikzpicture}
-
\begin{tikzpicture}[baseline = -.1cm]
	\draw (0,-.6)--(0,.6);
	\draw (.25,.35) arc (180:0:.1cm) -- (.45,-.35) arc (0:-180:.1cm);
	\node at (-.15,.5) {{\scriptsize{$3$}}};
	\node at (-.15,-.5) {{\scriptsize{$3$}}};
	\nbox{unshaded}{(0,0)}{.35}{0}{0}{$e_{4,+}^{2,2}$};
\end{tikzpicture}
\right) 
\in \spann(\trains_{3,+}(\cS_{4,+})).
$$
We have $\cup_0(T)\in \spann(\trains_{4,+}(\cS_{4,+}))$ by Lemma \ref{lem:Cup0InTrains}.
Thus $\cup_i(T)\in \spann(\trains_{4,+}(\cS_{4,+}))$ for all $0\leq i \leq 7$, and obviously $\cF(\check{f}^{(4)})\in  \spann(\trains_{4,+}(\cS_{4,+}))$.

To finish the proof, note that since $T,\cup_0(T)$ are in the span of trains, we have $A,B\in \spann(\trains_{4,\pm}(\cS_{4,\pm}))$ since $T$ is a rotational eigenvector, and $T\circ T$ and
$$
T\star T = 
\begin{tikzpicture}[baseline = -.1cm]
	\draw (-.1,.9)--(-.1,.4);
	\draw (0,.9)--(0,.4);
	\draw (.1,.9)--(.1,.4);
	\draw (.05,-.9)--(.05,.4);
	\draw (-.05,-.9)--(-.05,.4);
	\draw (-.5,.9)--(-.5,-.2) arc (-180:-30:.25cm);
	\draw (.5,-.9)--(.5,.2) arc (0:150:.25cm);
	\draw (.15,-.2) arc (180:0:.12cm) -- (.39,-.9);
	\ncircle{unshaded}{(0,.45)}{.3}{180}{$\check{T}$};
	\ncircle{unshaded}{(0,-.45)}{.3}{210}{$T$};
\end{tikzpicture}
$$
are concatenations of trains on $T$ and $\cup_0(T)$.
\end{proof}

\begin{cor}\label{cor:RotationInTrains}
$\cF(\cS_{4,\pm}),\cF^{-1}(\cS_{4,\pm})\subset \spann(\trains_{4,\mp}(\cS_{4,\mp}))$.
\end{cor}
\begin{proof}
First, $\cF(\cS_{4,\pm})\subseteq \spann(\cB_{4,\mp})\subseteq \spann(\trains_{4,\mp}(\cS_{4,\mp}))$ by Lemmas \ref{lem:1ClickOnS4} and \ref{lem:BInTrains}.
Since the train space is closed under adjoints, the same lemmas imply
$\cF^{-1}(\cS_{4,\pm})\subseteq \spann(\cB_{4,\mp}^*)\subseteq \spann(\trains_{4,\mp}(\cS_{4,\mp}))$.
\end{proof}

\begin{lem}\label{lem:TrainUp}
$\cS_{4,-}\subset \spann(\trains_{4,-}(\cS_{5,-}))$.
\end{lem}
\begin{proof}
We see that successive partial traces of the $e_{5,-}^{i,i}$ give nonzero multiples of $e_{4,-}^{i+1}$ and $e_{3,-}^i$ by Definitions \ref{defn:S5Minus} and Lemma \ref{lem:S4Minus},
$$
\begin{tikzpicture}[baseline = -.1cm]
	\draw (0,-.6)--(0,.6);
	\draw (.25,.35) arc (180:0:.1cm) -- (.45,-.35) arc (0:-180:.1cm);
	\node at (-.15,.5) {{\scriptsize{$4$}}};
	\node at (-.15,-.5) {{\scriptsize{$4$}}};
	\nbox{unshaded}{(0,0)}{.35}{0}{0}{$e_{5,-}^{i,i}$};
\end{tikzpicture}
=
\left(\frac{\sqrt{3+\sqrt{5}}}{\frac{1}{2}(3+\sqrt{5})}\right)
e_{4,-}^{i+1}
\text{ and }
\begin{tikzpicture}[baseline = -.1cm]
	\draw (0,-.6)--(0,.6);
	\draw (.25,.35) arc (180:0:.1cm) -- (.45,-.35) arc (0:-180:.1cm);
	\node at (-.15,.5) {{\scriptsize{$3$}}};
	\node at (-.15,-.5) {{\scriptsize{$3$}}};
	\nbox{unshaded}{(0,0)}{.35}{0}{0}{$e_{4,-}^{i+1}$};
\end{tikzpicture}
=
\left(\frac{\frac{1}{2}(3+\sqrt{5})}{\sqrt{7+3\sqrt{5}}}\right)
e_{3,-}^{i}
\,.
$$
Hence $e_{4,-}^{i+1}\in \spann(\trains_{4,-}(\cS_{5,-}))$ and $e_{3,-}^{i+1}\in \spann(\trains_{3,-}(\cS_{5,-}))$.
By Definition \ref{defn:gi},
$$
g_i = 
\begin{tikzpicture}[baseline = -.1cm]
	\draw (0,-.6)--(0,.6);
	\draw (.5,-.6)--(.5,.6);
	\node at (-.15,.5) {{\scriptsize{$3$}}};
	\node at (-.15,-.5) {{\scriptsize{$3$}}};
	\nbox{unshaded}{(0,0)}{.35}{0}{0}{$e_{3,-}^{i}$};
\end{tikzpicture}
-
\left(\frac{2+\sqrt{5}}{\sqrt{7+3\sqrt{5}}}\right)
e_{3,-}^i\circ e_{3,-}^i
\in \spann(\trains_{4,-}(\cS_{5,-})),
$$
and thus so are $g_i - e_{4,-}^{i+1}$, i.e., $e_{4,-}^1,e_{4,-}^4\in \spann(\trains_{5,-}(\cS_{5,-}))$.
\end{proof}

\begin{prop}\label{prop:1StrandPlus}
For all $s\in \cS_{4,+}$, 
$
\begin{tikzpicture}[baseline = -.1cm]
	\draw (0,-.6)--(0,.6);
	\fill[shaded] (-.5,-.6)--(-.5,.6) -- (-.65,.6) -- (-.65,-.6);
	\draw (-.5,-.6)--(-.5,.6);
	\node at (.15,.5) {{\scriptsize{$4$}}};
	\node at (.15,-.5) {{\scriptsize{$4$}}};
	\nbox{unshaded}{(0,0)}{.35}{0}{0}{$s$};
\end{tikzpicture}
\in \spann(\trains_{5,-}(\cS_{5,-})).
$
\end{prop}
\begin{proof}
By Lemma \ref{lem:4ClickRotation}, $\cF^4$ preserves $\spann(\cS_{4,+})$.
Applying $\cF^5$ to rotate by $\pi$, it suffices to show that for each $s\in \cS_{4,+}$,
$$
\begin{tikzpicture}[baseline = -.1cm]
	\draw (0,-.6)--(0,.6);
	\fill[shaded] (.5,-.6)--(.5,.6) -- (.65,.6) -- (.65,-.6);
	\draw (.5,-.6)--(.5,.6);
	\node at (-.15,.5) {{\scriptsize{$4$}}};
	\node at (-.15,-.5) {{\scriptsize{$4$}}};
	\nbox{unshaded}{(0,0)}{.35}{0}{0}{$s$};
\end{tikzpicture}
\in \cF^5(\spann(\trains_{5,-}(\cS_{5,-}))).
$$
By taking adjoints and multiplying relations, it suffices to consider $s=e_{4,+}^{1,2}$.
Now by Definition \ref{defn:5boxes}, we know that
$$
\begin{tikzpicture}[baseline = -.1cm]
	\draw (0,-.6)--(0,.6);
	\fill[shaded] (.5,-.6)--(.5,.6) -- (.65,.6) -- (.65,-.6);
	\draw (.5,-.6)--(.5,.6);
	\node at (-.15,.5) {{\scriptsize{$4$}}};
	\node at (-.15,-.5) {{\scriptsize{$4$}}};
	\nbox{unshaded}{(0,0)}{.35}{0}{0}{$e_{4,+}^{1,2}$};
\end{tikzpicture}
=
e_{5,+}^{1,2}+
\left(\frac{\sqrt{7+3\sqrt{5}}}{2+\sqrt{5}}\right)
\left(
(e_{4,+}^{1,1}+e_{4,+}^{1,2})\circ (e_{4,+}^{1,2}+e_{4,+}^{2,2})
\right).
$$
By Definition \ref{defn:S5Minus}, $\cF^5$ maps $\spann(\cS_{5,+})$ to $\spann(\cS_{5,-})$, so $\cF^5(e_{5,+}^{1,2})\in \spann(\cS_{5,-})$.
Hence it suffices to show that for each $s,t\in \cS_{4,+}$, $\cF^5(s\circ t) \in\spann(\trains_{5,-}(\cS_{5,-}))$.
Isotoping $\cF^5(s \circ t)$, we have
$$
\cF^5(s\circ t) = 
\cF^5\left(
\begin{tikzpicture}[baseline = -.1cm]
	\fill[shaded] (0,.15) -- (.15,.15) arc (-180:0:.1cm) -- (.35,.9) -- (.5,.9) -- (.5,-.9) -- (.35,-.9) -- (.35,-.15) arc (0:180:.1cm) -- (0,-.15);
	\draw (0,-.9)--(0,.9);
	\node at (-.15,.8) {{\scriptsize{$4$}}};
	\node at (-.15,-.8) {{\scriptsize{$4$}}};
	\node at (-.15,0) {{\scriptsize{$3$}}};
	\draw (.15,.15) arc (-180:0:.1cm) -- (.35,.9);
	\draw (.15,-.15) arc (180:0:.1cm) -- (.35,-.9);
	\nbox{unshaded}{(0,.4)}{.25}{0}{0}{$t$};
	\nbox{unshaded}{(0,-.4)}{.25}{0}{0}{$s$};
\end{tikzpicture}
\right)
=
\begin{tikzpicture}[baseline = -.1cm, xscale=-1]
	\fill[shaded] (0,.15) -- (.15,.15) arc (-180:0:.1cm) -- (.35,.9) -- (.5,.9) -- (.5,-.9) -- (.35,-.9) -- (.35,-.15) arc (0:180:.1cm) -- (0,-.15);
	\draw (0,-.9)--(0,.9);
	\node at (-.15,.8) {{\scriptsize{$4$}}};
	\node at (-.15,-.8) {{\scriptsize{$4$}}};
	\node at (-.15,0) {{\scriptsize{$3$}}};
	\draw (.15,.15) arc (-180:0:.1cm) -- (.35,.9);
	\draw (.15,-.15) arc (180:0:.1cm) -- (.35,-.9);
	\nbox{unshaded}{(0,.4)}{.25}{0}{0}{$s$};
	\nbox{unshaded}{(0,-.4)}{.25}{0}{0}{$t$};
\end{tikzpicture}
=
\begin{tikzpicture}[baseline = -.4cm]
	\fill[shaded] (-.5,-.9) -- (-.7,-.9) -- (-.7,.6) -- (1.7,.6) -- (1.7,-.9) -- (1.5,-.9) -- (1.2,0) -- (.8,0) -- (.8,-.6) arc (0:-180:.3cm) -- (.2,0) -- (-.2,0) -- (-.5,-.9);
	\draw (.2,0) -- (.2,-.6) arc (180:360:.3cm) -- (.8,0);
	\draw (-.2,0) -- (-.5,-.9);
	\draw (0,0) -- (0,-.9);
	\draw (1.2,0) -- (1.5,-.9);
	\draw (1,0) -- (1,-.9);
	\ncircle{unshaded}{(0,0)}{.3}{-120}{$t$}
	\ncircle{unshaded}{(1,0)}{.3}{-60}{$s$}
	\node at (.5,-1.05) {{\scriptsize{$3$}}};
	\node at (0,-1.05) {{\scriptsize{$4$}}};
	\node at (1,-1.05) {{\scriptsize{$4$}}};
	\draw[dashed] (-.45,-.45) rectangle (.45,.45);
	\draw[dashed] (.55,-.45) rectangle (1.45,.45);
\end{tikzpicture}\,.
$$
Now to get an element in the span of the trains, we need to rotate $s,t$ by one click.
The dashed boxes above may be replaced with  $\cF(s)$ and $\cF^{-1}(t)$, and $\cF(s),\cF^{-1}(t)\in \spann(\trains_{5,-}(\cS_{5,-}))$ by Corollary \ref{cor:RotationInTrains} and Lemma \ref{lem:TrainUp}.
\end{proof}

\begin{prop}\label{prop:1StrandMinus4}
For all $s\in \cS_{4,-}$, 
\,$
\begin{tikzpicture}[baseline = -.1cm]
	\fill[shaded] (-.5,-.6)--(-.5,.6) -- (.5,.6) -- (.5,-.6);
	\draw (0,-.6)--(0,.6);
	\draw (-.5,-.6)--(-.5,.6);
	\node at (.15,.5) {{\scriptsize{$4$}}};
	\node at (.15,-.5) {{\scriptsize{$4$}}};
	\nbox{unshaded}{(0,0)}{.35}{0}{0}{$s$};
\end{tikzpicture}
\in \spann(\trains_{5,+}(\cS_{4,+})).
$
\end{prop}
\begin{proof}
Applying $\cF^4$ to rotate by $\pi$ induces the duality on $\cS_{4,-}$, which we know by Lemma \ref{lem:S4Minus}.
Hence it suffices to show that for each $s\in \cS_{4,-}$,
$$
\begin{tikzpicture}[baseline = -.1cm]
	\fill[shaded] (-.5,-.6)--(-.5,.6) -- (.5,.6) -- (.5,-.6);
	\draw (0,-.6)--(0,.6);
	\draw (.5,-.6)--(.5,.6);
	\node at (-.15,.5) {{\scriptsize{$4$}}};
	\node at (-.15,-.5) {{\scriptsize{$4$}}};
	\nbox{unshaded}{(0,0)}{.35}{0}{0}{$s$};
\end{tikzpicture}\,
\in \cF^5(\spann(\trains_{5,+}(\cS_{4,+}))).
$$
By Definition \ref{defn:4minus}, for $e_{4,-}^2,e_{4,-}^3$, we have (when $i=1,2$)
$$
\begin{tikzpicture}[baseline = -.1cm]
	\fill[shaded] (-.5,-.6)--(-.5,.6) -- (.5,.6) -- (.5,-.6);
	\draw (0,-.6)--(0,.6);
	\draw (.5,-.6)--(.5,.6);
	\node at (.15,.5) {{\scriptsize{$4$}}};
	\node at (.15,-.5) {{\scriptsize{$4$}}};
	\nbox{unshaded}{(0,0)}{.35}{0}{0}{$e_{4,-}^{i+1}$};
\end{tikzpicture}\,
=
e_{5,-}^{i,i} +
\left(\frac{\sqrt{7+3\sqrt{5}}}{\frac{1}{2}(3+\sqrt{5})}\right)
\left(
e_{4,-}^{i+1}\circ e_{4,-}^{i+1}
\right).
$$
We know that $\cF^5(e_{5,-}^{i,i})\in \spann(\cS_{5,+})\subset \spann(\trains_{5,+}(\cS_{4,+}))$ by Lemma \ref{lem:5ClickRotation} and Definition \ref{defn:5boxes}.
By positivity and taking dimensions, for $s=e_{4,-}^1,e_{4,-}^4$, we have
$$
\begin{tikzpicture}[baseline = -.1cm]
	\fill[shaded] (-.5,-.6)--(-.5,.6) -- (.5,.6) -- (.5,-.6);
	\draw (0,-.6)--(0,.6);
	\draw (.5,-.6)--(.5,.6);
	\node at (.15,.5) {{\scriptsize{$4$}}};
	\node at (.15,-.5) {{\scriptsize{$4$}}};
	\nbox{unshaded}{(0,0)}{.35}{0}{0}{$s$};
\end{tikzpicture}\,
=
\left(\frac{\sqrt{7+3\sqrt{5}}}{\frac{1}{2}(1+\sqrt{5})}\right)
(s\circ s).
$$
We are finished by an argument similar to the end of the proof of Proposition \ref{prop:1StrandPlus}
\end{proof}

\begin{prop}\label{prop:1StrandMinus5}
For all $s\in \cS_{5,-}$, 
\,$
\begin{tikzpicture}[baseline = -.1cm]
	\fill[shaded] (-.5,-.6)--(-.5,.6) -- (0,.6) -- (0,-.6);
	\draw (0,-.6)--(0,.6);
	\draw (-.5,-.6)--(-.5,.6);
	\node at (.15,.5) {{\scriptsize{$5$}}};
	\node at (.15,-.5) {{\scriptsize{$5$}}};
	\nbox{unshaded}{(0,0)}{.35}{0}{0}{$s$};
\end{tikzpicture}
\in \spann(\trains_{6,+}(\cS_{4,+})).
$
\end{prop}
\begin{proof}
By Definition \ref{defn:S5Minus}, $\cF^5$ maps $\spann(\cS_{5,-})$ to $\spann(\cS_{5,+})$.
Again, by applying $\cF^5$ to rotate by $\pi$ and by taking adjoints and multiplying, it suffices to show that
$$
\begin{tikzpicture}[baseline = -.1cm]
	\draw (0,-.6)--(0,.6);
	\fill[shaded] (.5,-.6)--(.5,.6) -- (0,.6) -- (0,-.6);
	\draw (.5,-.6)--(.5,.6);
	\node at (-.15,.5) {{\scriptsize{$5$}}};
	\node at (-.15,-.5) {{\scriptsize{$5$}}};
	\nbox{unshaded}{(0,0)}{.35}{0}{0}{$e_{5,+}^{1,2}$};
\end{tikzpicture}\,
=
e_{6,+}^{1,2}
+
\left(\frac{2+\sqrt{5}}{\sqrt{3+\sqrt{5}}}\right)
\left(
e_{5,+}^{1,1}\circ e_{5,+}^{1,2}
\right)
\in \cF^6(\spann(\trains_{6,+}(\cS_{4,+}))).
$$
We know $\cF^6(e_{6,+}^{1,2})=-e_{6,+}^{2,1}\in \spann(\trains_{6,+}(\cS_{4,+}))$ by Lemma \ref{lem:6ClickRotation} and Definition \ref{defn:6boxes}.

It remains to examine $\cF^6(s\circ t)$ for $s,t\in \cS_{5,+}\subset \spann(\trains_{5,+}(\cS_{4,+}))$.
Now each $s,t\in \cS_{5,+}$ is a linear combination of trains on $\cS_{4,+}$, and expanding the formulas, we see each $s\circ t$ is a linear combination of diagrams given on the left hand side of Figure \ref{fig:RotateTrains}, where we just write $x$ for some element of $\cS_{4,+}$, and we don't require all the $x$'s to be the same.
Applying $\cF^6$ to rotate by $\pi$, we are left with the diagrams on the right hand side of Figure \ref{fig:RotateTrains}.

\begin{figure}[!ht]
$$
\begin{tikzpicture}[baseline = -.1cm, xscale=-1]
	\draw (0,.9)--(0,-.9);	
	\node at (.15,.8) {{\scriptsize{$4$}}};
	\node at (.15,0) {{\scriptsize{$4$}}};
	\node at (.15,-.8) {{\scriptsize{$4$}}};
	\filldraw[shaded] (-.25,.9) arc (360:180:.15cm);
	\filldraw[shaded] (-.25,-.9) arc (0:180:.15cm);
	\nbox{unshaded}{(0,-.4)}{.25}{0}{0}{$x$}
	\nbox{unshaded}{(0,.4)}{.25}{0}{0}{$x$}
\end{tikzpicture}
\,,\,
\begin{tikzpicture}[baseline = -.1cm, xscale=-1]
	\draw (0,.9)--(0,.4);
	\draw (0,-.9)--(0,-.4);		
	\node at (.15,.8) {{\scriptsize{$4$}}};
	\node at (.15,0) {{\scriptsize{$3$}}};
	\node at (.15,-.8) {{\scriptsize{$4$}}};
	\filldraw[shaded] (-.25,.9) arc (360:180:.15cm);
	\filldraw[shaded] (-.35,-.9) -- (-.35,-.15) arc (180:0:.1cm) -- (-.15,-.4) -- (0,-.4) -- (0,-.4) -- (0,.4) -- (-.15,.15) arc (0:-90:.1cm) arc (90:180:.25cm) -- (-.5,-.9);
	\nbox{unshaded}{(0,-.4)}{.25}{0}{0}{$x$}
	\nbox{unshaded}{(0,.4)}{.25}{0}{0}{$x$}
\end{tikzpicture}
\,,\,
\begin{tikzpicture}[baseline = -.1cm, yscale=-1, xscale=-1]
	\draw (0,.9)--(0,.4);
	\draw (0,-.9)--(0,-.4);		
	\node at (.15,.8) {{\scriptsize{$4$}}};
	\node at (.15,0) {{\scriptsize{$3$}}};
	\node at (.15,-.8) {{\scriptsize{$4$}}};
	\filldraw[shaded] (-.25,.9) arc (360:180:.15cm);
	\filldraw[shaded] (-.35,-.9) -- (-.35,-.15) arc (180:0:.1cm) -- (-.15,-.4) -- (0,-.4) -- (0,-.4) -- (0,.4) -- (-.15,.15) arc (0:-90:.1cm) arc (90:180:.25cm) -- (-.5,-.9);
	\nbox{unshaded}{(0,-.4)}{.25}{0}{0}{$x$}
	\nbox{unshaded}{(0,.4)}{.25}{0}{0}{$x$}
\end{tikzpicture}
\,,\,
\begin{tikzpicture}[baseline = -.1cm, xscale=-1]
	\draw (0,1.3)--(0,.4);
	\draw (0,-1.3)--(0,-.4);		
	\node at (.15,1.2) {{\scriptsize{$4$}}};
	\node at (.15,.4) {{\scriptsize{$3$}}};
	\node at (.15,-.4) {{\scriptsize{$3$}}};
	\node at (.15,-1.2) {{\scriptsize{$4$}}};
	\filldraw[shaded] (-.35,-1.3) -- (-.35,-.55) arc (180:0:.1cm) -- (-.15,-.8) -- (0,-.8) -- (0,-.8) -- (0,0) -- (-.15,-.25) arc (0:-90:.1cm) arc (90:180:.25cm) -- (-.5,-1.3);
	\filldraw[shaded] (-.35,1.3) -- (-.35,.55) arc (-180:0:.1cm) -- (-.15,.8) -- (0,.8) -- (0,.8) -- (0,0) -- (-.15,.25) arc (0:90:.1cm) arc (-90:-180:.25cm) -- (-.5,1.3);
	\nbox{unshaded}{(0,-.8)}{.25}{0}{0}{$x$}
	\nbox{unshaded}{(0,0)}{.25}{0}{0}{$x$}
	\nbox{unshaded}{(0,.8)}{.25}{0}{0}{$x$}
\end{tikzpicture}
\,\,
\begin{tikzpicture}[baseline = -.1cm]
	\draw (0,0)--(2,0);
	\draw (0,.1)--(0,-.1);
	\draw (1.8,.1) -- (2,0) -- (1.8,-.1);
	\node at (1,.3) {$\cF^6$};		
\end{tikzpicture}
\,\,
\begin{tikzpicture}[baseline = -.1cm]
	\draw (0,.9)--(0,-.9);	
	\node at (.15,.8) {{\scriptsize{$4$}}};
	\node at (.15,0) {{\scriptsize{$4$}}};
	\node at (.15,-.8) {{\scriptsize{$4$}}};
	\filldraw[shaded] (-.25,.9) arc (360:180:.15cm);
	\filldraw[shaded] (-.25,-.9) arc (0:180:.15cm);
	\nbox{unshaded}{(0,-.4)}{.25}{0}{0}{$x$}
	\nbox{unshaded}{(0,.4)}{.25}{0}{0}{$x$}
\end{tikzpicture}
\,,\,
\begin{tikzpicture}[baseline = -.1cm]
	\draw (0,.9)--(0,.4);
	\draw (0,-.9)--(0,-.4);		
	\node at (.15,.8) {{\scriptsize{$4$}}};
	\node at (.15,0) {{\scriptsize{$3$}}};
	\node at (.15,-.8) {{\scriptsize{$4$}}};
	\filldraw[shaded] (-.25,.9) arc (360:180:.15cm);
	\filldraw[shaded] (-.35,-.9) -- (-.35,-.15) arc (180:0:.1cm) -- (-.15,-.4) -- (0,-.4) -- (0,-.4) -- (0,.4) -- (-.15,.15) arc (0:-90:.1cm) arc (90:180:.25cm) -- (-.5,-.9);
	\nbox{unshaded}{(0,-.4)}{.25}{0}{0}{$x$}
	\nbox{unshaded}{(0,.4)}{.25}{0}{0}{$x$}
\end{tikzpicture}
\,,\,
\begin{tikzpicture}[baseline = -.1cm, yscale=-1]
	\draw (0,.9)--(0,.4);
	\draw (0,-.9)--(0,-.4);		
	\node at (.15,.8) {{\scriptsize{$4$}}};
	\node at (.15,0) {{\scriptsize{$3$}}};
	\node at (.15,-.8) {{\scriptsize{$4$}}};
	\filldraw[shaded] (-.25,.9) arc (360:180:.15cm);
	\filldraw[shaded] (-.35,-.9) -- (-.35,-.15) arc (180:0:.1cm) -- (-.15,-.4) -- (0,-.4) -- (0,-.4) -- (0,.4) -- (-.15,.15) arc (0:-90:.1cm) arc (90:180:.25cm) -- (-.5,-.9);
	\nbox{unshaded}{(0,-.4)}{.25}{0}{0}{$x$}
	\nbox{unshaded}{(0,.4)}{.25}{0}{0}{$x$}
\end{tikzpicture}
\,,\,
\begin{tikzpicture}[baseline = -.1cm]
	\draw (0,1.3)--(0,.4);
	\draw (0,-1.3)--(0,-.4);		
	\node at (.15,1.2) {{\scriptsize{$4$}}};
	\node at (.15,.4) {{\scriptsize{$3$}}};
	\node at (.15,-.4) {{\scriptsize{$3$}}};
	\node at (.15,-1.2) {{\scriptsize{$4$}}};
	\filldraw[shaded] (-.35,-1.3) -- (-.35,-.55) arc (180:0:.1cm) -- (-.15,-.8) -- (0,-.8) -- (0,-.8) -- (0,0) -- (-.15,-.25) arc (0:-90:.1cm) arc (90:180:.25cm) -- (-.5,-1.3);
	\filldraw[shaded] (-.35,1.3) -- (-.35,.55) arc (-180:0:.1cm) -- (-.15,.8) -- (0,.8) -- (0,.8) -- (0,0) -- (-.15,.25) arc (0:90:.1cm) arc (-90:-180:.25cm) -- (-.5,1.3);
	\nbox{unshaded}{(0,-.8)}{.25}{0}{0}{$x$}
	\nbox{unshaded}{(0,0)}{.25}{0}{0}{$x$}
	\nbox{unshaded}{(0,.8)}{.25}{0}{0}{$x$}
\end{tikzpicture}
$$
\caption{$s\circ t$ for $s,t\in \cS_{5,+}$ in terms of trains on $\cS_{4,+}$, and their 6-click rotations under $\cF^6$}
\label{fig:RotateTrains}
\end{figure}

To see the last three diagrams are really linear combinations of trains, we just use the formulas for $\cF,\cF^{-1}$ for elements of $\cS_{4,+}$ in terms of trains on $\cS_{4,-}$ afforded by Corollary \ref{cor:RotationInTrains}. Finally, we use Proposition \ref{prop:1StrandMinus4} to pull the elements of $\cS_{4,-}$ through the remaining strand.
\end{proof}

\begin{thm}\label{thm:2Strand}
We have 2-strand relations for the elements of $\cS_{4,+}$, which are independent of the parameter $\lambda\in\bbT$.
\end{thm}
\begin{proof}
We must show that for each $s\in \cS_{4,+}$, we have
$$
\begin{tikzpicture}[baseline = -.1cm]
	\fill[shaded] (-.5,-.6)--(-.5,.6) -- (-.65,.6) -- (-.65,-.6);
	\draw (0,-.6)--(0,.6);
	\draw (-.5,-.6)--(-.5,.6);
	\draw (-.65,-.6)--(-.65,.6);
	\node at (.15,.5) {{\scriptsize{$4$}}};
	\node at (.15,-.5) {{\scriptsize{$4$}}};
	\nbox{unshaded}{(0,0)}{.35}{0}{0}{$s$};
\end{tikzpicture}
\in \spann(\trains_{6,+}(\cS_{4,+})).
$$
First, by Proposition \ref{prop:1StrandPlus}, we have
$$
\begin{tikzpicture}[baseline = -.1cm]
	\draw (0,-.6)--(0,.6);
	\fill[shaded] (-.5,-.6)--(-.5,.6) -- (-.65,.6) -- (-.65,-.6);
	\draw (-.5,-.6)--(-.5,.6);
	\node at (.15,.5) {{\scriptsize{$4$}}};
	\node at (.15,-.5) {{\scriptsize{$4$}}};
	\nbox{unshaded}{(0,0)}{.35}{0}{0}{$s$};
\end{tikzpicture}
\in \spann(\trains_{5,-}(\cS_{5,-})).
$$
For each $s'\in \cS_{5,-}$, by Proposition  \ref{prop:1StrandMinus5}, we have
$$
\begin{tikzpicture}[baseline = -.1cm]
	\fill[shaded] (-.5,-.6)--(-.5,.6) -- (0,.6) -- (0,-.6);
	\draw (0,-.6)--(0,.6);
	\draw (-.5,-.6)--(-.5,.6);
	\node at (.15,.5) {{\scriptsize{$5$}}};
	\node at (.15,-.5) {{\scriptsize{$5$}}};
	\nbox{unshaded}{(0,0)}{.35}{0}{0}{$s'$};
\end{tikzpicture}
\in \spann(\trains_{6,+}(\cS_{4,+})).
$$
The relations from Propositions  \ref{prop:1StrandPlus} and \ref{prop:1StrandMinus5} are independent of $\lambda$.
\end{proof}

\begin{thm}\label{thm:Evaluable}
$\cP_\bullet$ is evaluable, and the value of any closed diagram is independent of the parameter $\lambda\in \bbT$.
\end{thm}
\begin{proof}
We use Theorem \ref{thm:EvaluableSubalgebra}, with $n=4$, and the generating set $\cS_{4,+}$.
As at the end of Section \ref{sec:StabilityAndJellyfish}, we let $\cS_{j,+}=\emptyset$ for $j<3$, and for $j=3$, we let $\cS_{3,+}=\{T\}$.
We know that $T$ is uncappable and $T^2=\jw{3}$ by Lemma \ref{lem:ParameterFamilies}.

Now the projections in $\cS_{4,+}$ are defined from the $e_{3,+}^i=(\jw{3}\pm T)/2$ via the generalized Wenzl relation.
Hence we know all annular maps applied to $\cS_{4,+}$, since its elements are rectangularly uncappable, the $e_{4,+}^{1,2},e_{4,+}^{2,1}$ are uncappable, and the partial traces of the $e_{4,+}^{i,i}$ are given in Definition \ref{defn:4boxes}.
Moreover, we also know that the elements of $\cS_{4,+}$ absorb $T$:
$$
\begin{tikzpicture}[baseline = -.1cm, xscale=-1]
	\draw (0,1.3)--(0,.4);
	\draw (0,-1.3)--(0,-.4);		
	\node at (.15,1.2) {{\scriptsize{$3$}}};
	\node at (.15,0) {{\scriptsize{$3$}}};
	\node at (.15,-1.2) {{\scriptsize{$4$}}};
	\filldraw[shaded] (0,1.3) -- (0,-.35) -- (-.25,-.35) -- (-.25,-.25) arc (0:90:.15cm) arc (-90:-180:.25cm) -- (-.65,1.3);
	\nbox{unshaded}{(0,-.6)}{.35}{0}{0}{$e_{4,+}^{i,j}$}
	\nbox{unshaded}{(0,.6)}{.35}{0}{0}{$T$}
\end{tikzpicture}
=
\begin{tikzpicture}[baseline = -.1cm, xscale=-1]
	\draw (0,1.3)--(0,.4);
	\draw (0,-1.3)--(0,-.4);		
	\node at (.15,1.2) {{\scriptsize{$3$}}};
	\node at (.15,0) {{\scriptsize{$3$}}};
	\node at (.15,-1.2) {{\scriptsize{$4$}}};
	\filldraw[shaded] (0,1.3) -- (0,-.35) -- (-.25,-.35) -- (-.25,-.25) arc (0:90:.15cm) arc (-90:-180:.25cm) -- (-.65,1.3);
	\nbox{unshaded}{(0,-.6)}{.35}{0}{0}{$e_{4,+}^{i,j}$}
	\nbox{unshaded}{(0,.6)}{.35}{0}{0}{$e_{3,+}^1$}
\end{tikzpicture}
-
\begin{tikzpicture}[baseline = -.1cm, xscale=-1]
	\draw (0,1.3)--(0,.4);
	\draw (0,-1.3)--(0,-.4);		
	\node at (.15,1.2) {{\scriptsize{$3$}}};
	\node at (.15,0) {{\scriptsize{$3$}}};
	\node at (.15,-1.2) {{\scriptsize{$4$}}};
	\filldraw[shaded] (0,1.3) -- (0,-.35) -- (-.25,-.35) -- (-.25,-.25) arc (0:90:.15cm) arc (-90:-180:.25cm) -- (-.65,1.3);
	\nbox{unshaded}{(0,-.6)}{.35}{0}{0}{$e_{4,+}^{i,j}$}
	\nbox{unshaded}{(0,.6)}{.35}{0}{0}{$e_{3,+}^2$}
\end{tikzpicture}
=
\delta_{j,1}
\begin{tikzpicture}[baseline = -.1cm]
	\draw (0,-.6)--(0,.6);
	\node at (-.15,.5) {{\scriptsize{$4$}}};
	\node at (-.15,-.5) {{\scriptsize{$4$}}};
	\nbox{unshaded}{(0,0)}{.35}{0}{0}{$e_{4,+}^{i,1}$};
\end{tikzpicture}
-
\delta_{j,2}
\begin{tikzpicture}[baseline = -.1cm]
	\draw (0,-.6)--(0,.6);
	\node at (-.15,.5) {{\scriptsize{$4$}}};
	\node at (-.15,-.5) {{\scriptsize{$4$}}};
	\nbox{unshaded}{(0,0)}{.35}{0}{0}{$e_{4,+}^{i,2}$};
\end{tikzpicture}\,.
$$
By Lemma \ref{lem:S4Plus}, $\cS_{4,+}$ is a system of matrix units, so $\cP_\bullet$ is evaluable by Theorem \ref{thm:EvaluableSubalgebra}.

The 2-strand relations from Theorem \ref{thm:2Strand} and the annular and absorption maps are independent of $\lambda\in\bbT$, so the value of any closed diagram is also independent of $\lambda$.
\end{proof}

%%%%%%%%%%%%%%%%%%%%%%%%%%%%%%%%%%%%%%%%%%%%%%%%%%%%%%%%%%%%%%%%%%%
\subsection{The principal graphs}\label{sec:PrincipalGraphs}

\begin{thm}\label{thm:PrincipalGraphs}
The subfactor planar algebra $\cP_\bullet$ has principal graphs
$$
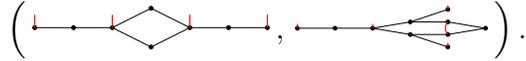
\left(
\bigraph{bwd1v1v1p1v1x1v1v1duals1v1v1v1},
\bigraph{bwd1v1v1p1v1x0p1x0p0x1p0x1v0x1x1x0duals1v1v1x3x2x4}
\right).
$$
\end{thm}
\begin{proof}
By Theorem \ref{thm:2Strand}, trains on $\cS_{4,+}$ span $\cP_{n,+}$ for all $n\geq 0$.
Since the elements of $\cS_{4,+}$ are rectangularly uncappable, and we know their partial traces by Definition \ref{defn:4boxes}, any reduced train on $\cS_{4,+}$ in $\cP_{3,+}$ is a reduced train on $T$.
Since $T$ is uncappable and $T^2=\jw{3}$, the reduced trains on $T$ in $\cP_{3,+}$ span $\cT\cL_{3,+}\oplus \bbC T$.
Hence $\cP_\bullet$ is 2-supertransitive with excess 1.

The principal graph is 2D2 by Lemma \ref{lem:S4Plus} and Definitions \ref{defn:5boxes} and \ref{defn:6boxes}.
The dual graph is as claimed by Lemma \ref{lem:S4Minus} and Definition \ref{defn:S5Minus}.
\end{proof}

We now prove our main Theorems \ref{thm:2D2} and \ref{thm:Unique}.

\begin{proof}[Proof of Theorems \ref{thm:2D2} and \ref{thm:Unique}]
Any subfactor planar algebra with principal graph 2D2 must embed in its graph planar algebra by \cite{MR2812459,gpa}.
By Lemma \ref{lem:ParameterFamilies}, there are exactly two 1-parameter family of solutions of the equation $T^2=\jw{3}$, which must hold when the principal graph is 2D2.
By Theorem \ref{thm:Evaluable}, the value of any closed diagram is independent of the parameter, so the generator always generates the same subfactor planar algebra, which has principal graphs as claimed in Theorem \ref{thm:2D2} by Theorem \ref{thm:PrincipalGraphs}.
Hence no subfactors with the other principal graphs exist.
\end{proof}

\begin{proof}[Proof of Corollary \ref{cor:Equi}]
The map $T\mapsto -T$ extends to a planar algebra automorphism of $\cP_\bullet$, which performs the following swaps: $e_{3,\pm}^1\leftrightarrow e_{3,\pm}^2$, $e^{1,1}_{4,+}\leftrightarrow e_{2,2}$, $e^{1,2}_{4,+}\leftrightarrow e^{2,1}_{4,+}$, $e_{4,-}^1\leftrightarrow e_{4,-}^4$, and $e_{4,-}^2\leftrightarrow e_{4,-}^3$, and similar appropriate swaps for $\cS_{5,\pm}$ and $\cS_{6,+}$.
However, $A$ and $B$ are fixed, since they are given by linear combinations of diagrams with even numbers of $T$'s in Definition \ref{defn:4boxes}.
Thus the map preserves the 2-strand jellyfish relations, since swapping the $e_{i,j}^{4,+}$ results in the same relations in the various lemmas, after using the map $T\mapsto -T$.
For example, swapping the appropriate rows in the matrixes in Figures \ref{fig:1ClickRotationMatrices} and \ref{fig:Cup0InTrains} and replacing $T$ with $-T$ gives the exact same matrices (see Remark \ref{rem:SameRelations}). 

Thus the fixed points of this automorphism form a 3-supertransitive subfactor planar algebra which contains Izumi's $3^{\bbZ/4\bbZ}$ subfactor planar algebra by Lemma \ref{lem:3ZMod4}.
Hence the fixed points are exactly this subfactor planar algebra.
\end{proof}

%%%%%%%%%%%%%%%%%%%%%%%%%%%%%%%%%%%%%%%%%%%%%%%%%%%%%%%%%%%%%%%%%%%
\section{$3^{\bbZ/2\bbZ\times \bbZ/2\bbZ}$ and 4442}\label{sec:4442}

We conclude with a brief section on the $3^{\bbZ/2\bbZ\times \bbZ/2\bbZ}$  and 4442 subfactor planar algebras, which are also related by equivariantization.

The $3^{\bbZ/2\bbZ\times \bbZ/2\bbZ}$ subfactor planar algebra is generated as a planar algebra by the 3 minimal projections at depth 4.
There is a $\bbZ/3\bbZ$-action corresponding to cyclicly permuting these projections, and the fixed points of this action forms the 4442 subfactor planar algebra constructed in \cite{1208.3637}.
As this result is due to Izumi \cite{IzumiUnpublished}, we simply write the formulas for the uncappable rotational eigenvectors for 4442 in terms of the minimal projections $P,Q,R$ of  $3^{\bbZ/2\bbZ\times \bbZ/2\bbZ}$ at depth 4 given in the beginning of the proof of \cite[Theorem 5.9]{1208.3637} listed from top to bottom respectively.
Define the following numbers:
\begin{align*}
\alpha &= -1-\frac{\sqrt{5}}{3} 
&
\beta &= \frac{1}{6} \left(3+\sqrt{5}\right)+i \sqrt{\frac{1}{10} \left(5+\sqrt{5}\right)}.
\end{align*}
Then an uncappable rotational eigenvector with eigenvalue $\exp\left(2\pi i\left(\frac{3}{5}\right)\right)$ at depth 5 of $3^{\bbZ/2\bbZ\times \bbZ/2\bbZ}$ is given by
$$
\left(\sqrt{\frac{3}{5}+\frac{1}{\sqrt{5}}}\right)\jw{5}
+\alpha(P\circ P + Q\circ Q + R\circ R)
+\beta(P\circ Q + Q\circ R + R\circ P)
+\overline{\beta}(P\circ R + Q\circ P + R\circ Q),
$$
and an another uncappable rotational eigenvector with eigenvalue $\exp\left(2\pi i\left(\frac{2}{5}\right)\right)$ is given by the same formula, but swapping $\beta$ and $\overline{\beta}$ (resulting in the conjugate eigenvalue).
Note that these elements are obviously fixed under cyclic permutations of $P,Q,R$.
These two elements generate the 4442 subfactor planar algebra, but they are non-trivial scalar multiples of the generators for 4442 given in \cite{1208.3637}.
They are also complex-conjugate to each other (entry-wise, as elements of the graph planar algebra).

Uniqueness of the $3^{\bbZ/2\bbZ\times \bbZ/2\bbZ}$ subfactor planar algebra follows from the uniqueness of Izumi's Cuntz-algebra model.
We should thus expect uniqueness of the 4442 subfactor planar algebra, but the Cuntz algebra model no longer works, as the even part of 4442 has more than 2 orbits under the $\bbZ/3\bbZ$-action.
However, there is a unique bi-unitary connection up to graph automorphism.

We now prove Theorem \ref{thm:4442}, i.e., that there is exactly one subfactor planar algebra with principal graphs 4442.
The proof of this theorem is similar to the proof of uniqueness for 3311 from \cite{MR2993924}, and we give only the most important details below.
This proof was also known to Izumi.

\begin{proof}[Proof of Theorem \ref{thm:4442}]
The diagrammatic branch matrix for 4442 with its 9 unknown phases is given by
$$
\frac{1}{A}
\left(
\begin{array}{cccc}
 -1& B & B & \sqrt{BC} \\
B & B \xi_1 & \xi_2 & \sqrt{BC} \eta_1 \\
B & \xi_3 & B \xi_4 & \sqrt{BC} \eta_2 \\
 \sqrt{BC} & \sqrt{BC} \eta_3 & \sqrt{BC} \eta_4 & 3 \zeta
\end{array}
\right)
$$
where $A=[5]$, $B=[4]$, $C=3[2]$, and $[2]=\sqrt{3+\sqrt{5}}$.
By a computation similar to \cite[Lemmas 3.1 and 3.4]{MR2993924}, but more difficult, there are exactly two solutions for these 9 unknown phases, which are complex conjugate to each other (so we list only one below):
\begin{align*}
\xi_1 &= \sqrt{-\frac{11}{16}-\frac{3 i \sqrt{15}}{16}}  \approx  0.395285-0.918559 i 
\\
\xi_2 &=\xi_3=-\zeta= -1
%\\
%\xi_3 &= -1
\\
\xi_4 &=  \sqrt{-\frac{11}{16}+\frac{3 i \sqrt{15}}{16}} \approx 0.395285+0.918559 i
\\
\eta_1 &=\eta_3= -\frac{1}{2} \sqrt{\frac{1}{2} \left(-1-3 \sqrt{5}-i\sqrt{6 \left(3-\sqrt{5}\right)}\right)} \approx -0.135045+0.990839 i
\\
\eta_2 &=\eta_4=  -\frac{1}{2} \sqrt{\frac{1}{2} \left(-1-3 \sqrt{5}+i \sqrt{6 \left(3-\sqrt{5}\right)}\right)} \approx -0.135045-0.990839 i
%\\
%\eta_3 &=  -\frac{1}{2} \sqrt{\frac{1}{2} \left(-1-3 \sqrt{5}-i \sqrt{6 \left(3-\sqrt{5}\right)}\right)} \approx -0.135045+0.990839 i  
%\\
%\eta_4 &= -\frac{1}{2} \sqrt{\frac{1}{2} \left(-1-3 \sqrt{5}+i \sqrt{6 \left(3-\sqrt{5}\right)}\right)} \approx -0.135045-0.990839 i
%\\
%\zeta&= 1
\end{align*}
Since $\overline{\xi_1}=\xi_4$ and $\overline{\eta_1}=\overline{\eta_3}=\eta_2=\eta_4$, the resulting diagrammatic branch matrices are mapped to each other by swapping the second and third rows and columns, which corresponds to a graph automorphism of 4442.
Since every bi-unitary connection on 4442 is gauge equivalent to one of these, there is exactly one gauge orbit of bi-unitary connections on 4442 up to graph automorphism, and thus there is at most one 4442 subfactor planar algebra.
Existence was shown in \cite{1208.3637}.
\end{proof}

\begin{remark}
Even though computing the phases is more difficult for 4442 than for 3311, we only get a discrete set of bi-unitary connections, not a one parameter family, which gives us uniqueness without using the $U^TU$ condition of \cite[Theorem 1.7]{MR2993924}.
\end{remark}

%%%%%%%%%%%%%%%%%%%%%%%%%%%%%%%%%%%%%%%%%%%%%%%%%%%%%%%%%%%%%%%%%%%%%
%%%%%%%%%%%%%%%%%%%%%%%%%%%%%%%%%%%%%%%%%%%%%%%%%%%%%%%%%%%%%%%%%%%%%
%%%%%%%%%%%%%%%%%%%%%%%%%%%%%%%%%%%%%%%%%%%%%%%%%%%%%%%%%%%%%%%%%%%%%

\begin{landscape}
\appendix
\section{One click rotations $\cF:\spann(\cS_{4,\pm})\to \spann(\cB_{4,\mp})$, and $\trains_{4,+}(\cS_{4,+})=\cP_{4,+}$}\label{sec:1ClickRotation}

\thispagestyle{empty}

\begin{figure}[!ht]
{\scriptsize{
\begin{align*}
\cF
\begin{pmatrix}
e_{4,+}^{1,1}\\
e_{4,+}^{1,2}\\
e_{4,+}^{2,1}\\
e_{4,+}^{2,2}
\end{pmatrix}
&=
\left(
\begin{array}{ccccccccccc}
 0 & -\frac{1}{6} & \frac{1}{3} & -\frac{1}{8} \sqrt{7-3 \sqrt{5}} & 0 & -\frac{1}{8} \sqrt{7-3 \sqrt{5}} & \frac{1}{8} \left(\sqrt{5}-1\right) & -\frac{1}{8} \sqrt{3+\sqrt{5}} & \frac{1}{2} & -\frac{1}{8}
   \sqrt{3+\sqrt{5}} & \frac{1}{8} \left(\sqrt{5}-1\right) \\
 \frac{i}{2} & 0 & 0 & -\frac{1}{8} i \sqrt{5 \left(\sqrt{5}-1\right)} & \frac{1}{4} i \sqrt{2+\sqrt{5}} & -\frac{1}{8} i \sqrt{5 \left(\sqrt{5}-1\right)} & \frac{1}{4} i \sqrt{\frac{1}{2}\left(\sqrt{5}-1\right)} &
   -\frac{1}{8} i \sqrt{\sqrt{5}-1} & \frac{1}{4} i \sqrt{\sqrt{5}-2} & -\frac{1}{8} i \sqrt{\sqrt{5}-1} & \frac{1}{4} i \sqrt{\frac{1}{2}\left(\sqrt{5}-1\right)} \\
 \frac{i}{2} & 0 & 0 & \frac{1}{8} i \sqrt{5 \left(\sqrt{5}-1\right)} & -\frac{1}{4} i \sqrt{2+\sqrt{5}} & \frac{1}{8} i \sqrt{5 \left(\sqrt{5}-1\right)} & -\frac{1}{4} i \sqrt{\frac{1}{2}
   \left(\sqrt{5}-1\right)} & \frac{1}{8} i \sqrt{\sqrt{5}-1} & -\frac{1}{4} i \sqrt{\sqrt{5}-2} & \frac{1}{8} i \sqrt{\sqrt{5}-1} & -\frac{1}{4} i \sqrt{\frac{1}{2} \left(\sqrt{5}-1\right)} \\
 0 & -\frac{1}{6} & \frac{1}{3} & \frac{1}{8} \sqrt{7-3 \sqrt{5}} & 0 & \frac{1}{8} \sqrt{7-3 \sqrt{5}} & \frac{1}{8} \left(1-\sqrt{5}\right) & \frac{1}{8}\sqrt{3+\sqrt{5}} & -\frac{1}{2} &
   \frac{1}{8}\sqrt{3+\sqrt{5}} & \frac{1}{8} \left(1-\sqrt{5}\right)
\end{array}
\right)
\cB_{4,-}
\\
\cF
\begin{pmatrix}
e_{4,-}^{1}\\
e_{4,-}^{2}\\
e_{4,-}^{3}\\
e_{4,-}^{4}
\end{pmatrix}
&=
\left(
\begin{array}{ccccccccccc}
 \frac{1}{2} i \sqrt{\frac{1}{2} \left(\sqrt{5}-2\right)} & \frac{1}{6 \sqrt{2}} & \frac{1}{6} \left(3-\sqrt{5}\right) & -\frac{1}{4 \sqrt{2}} & \frac{1}{4} & -\frac{1}{4 \sqrt{2}} & \frac{1}{8}
   \left(\sqrt{5}-1\right) & -\frac{1}{4 \sqrt{2}} & \frac{1}{4} & -\frac{1}{4 \sqrt{2}} & \frac{1}{8} \left(\sqrt{5}-1\right) \\
 \frac{1}{4} i \sqrt{5 \sqrt{5}-11} & \frac{\sqrt{5}-5}{12 \sqrt{2}} & \frac{1}{6} \left(\sqrt{5}-1\right) & \frac{1}{8}\sqrt{3-\sqrt{5}} & -\frac{1}{4} & \frac{1}{8}\sqrt{3-\sqrt{5}} & 0 & -\frac{1}{8}
   \sqrt{3-\sqrt{5}} & \frac{1}{4} & -\frac{1}{8} \sqrt{3-\sqrt{5}} & 0 \\
 \frac{1}{4} i \sqrt{5 \sqrt{5}-11} & \frac{\sqrt{5}-5}{12 \sqrt{2}} & \frac{1}{6} \left(\sqrt{5}-1\right) & -\frac{1}{8} \sqrt{3-\sqrt{5}} & \frac{1}{4} & -\frac{1}{8} \sqrt{3-\sqrt{5}} & 0 &
   \frac{1}{8}\sqrt{3-\sqrt{5}} & -\frac{1}{4} & \frac{}{8}\sqrt{3-\sqrt{5}} & 0 \\
 \frac{1}{2} i \sqrt{\frac{1}{2} \left(\sqrt{5}-2\right)} & \frac{1}{6 \sqrt{2}} & \frac{1}{6} \left(3-\sqrt{5}\right) & \frac{1}{4 \sqrt{2}} & -\frac{1}{4} & \frac{1}{4 \sqrt{2}} & \frac{1}{8}
   \left(1-\sqrt{5}\right) & \frac{1}{4 \sqrt{2}} & -\frac{1}{4} & \frac{1}{4 \sqrt{2}} & \frac{1}{8} \left(1-\sqrt{5}\right)
\end{array}
\right)
\cB_{4,+}
\end{align*}
}}
\caption{The one click rotation $\cF$ on $\cS_{4,\pm}$, from Lemma \ref{lem:1ClickOnS4}}
\label{fig:1ClickRotationMatrices}
\end{figure}

\begin{figure}[!ht]
\begin{align*}
\cup_0(T)
&=
\left(
\begin{array}{ccccccccccc}
2 \sqrt{5}&
2 i \sqrt{\sqrt{5}-2}&
-2 i \sqrt{\sqrt{5}-2}&
-2 \sqrt{5}&
0&
-1&
\sqrt{3+\sqrt{5}}&
-2-\sqrt{5}&
\sqrt{3+\sqrt{5}}&
-1&
0
\end{array}
\right)
\left(\cS_{4,+}\cup \set{\cup_i(T)}{1\leq i\leq 7}\right)
\\
\cup_0(\check{T})
&=
\left(
\begin{array}{ccccccccccc}
1+\sqrt{5}&
3+\sqrt{5}&
-3-\sqrt{5}&
-1-\sqrt{5}&
0&
-1&
\sqrt{3+\sqrt{5}}&
-2-\sqrt{5}&
\sqrt{3+\sqrt{5}}&
-1&
0
\end{array}
\right)
\left(\cS_{4,-}\cup \set{\cup_i(\check{T})}{1\leq i\leq 7}\right)
\end{align*}
\caption{$\cup_0(T)$ and $\cup_0(\check{T})$ are in the span of trains on $\cS_{4,\pm}$, from Lemma \ref{lem:Cup0InTrains}}
\label{fig:Cup0InTrains}
\end{figure}

\begin{remark}\label{rem:SameRelations}
The map $T\mapsto -T$ swaps $e_{4,+}^{1,1}\leftrightarrow e_{4,+}^{2,2}$, $e_{4,+}^{1,2}\leftrightarrow e_{4,+}^{2,1}$, $e_{4,-}^1\leftrightarrow e_{4,-}^4$, and $e_{4,-}^2\leftrightarrow e_{4,-}^3$.
This map preserves the equations in Figures \ref{fig:1ClickRotationMatrices} and \ref{fig:Cup0InTrains}.
For Figure \ref{fig:1ClickRotationMatrices}, it corresponds to reversing the order of the 4 rows on both sides, and then negating the last 7 entries of each row corresponding to the $\cup_i(T),\cup_i(\check{T})$ for $1\leq i\leq 8$.
For Figure \ref{fig:Cup0InTrains}, it corresponds to reversing the first 4 columns of the right hand side and fixing the final 7 entries, since all of $\cup_0(T),\cup_0(\check{T})$ and the rest of the $\cup_i(T),\cup_i(\check{T})$ for $1\leq i\leq 8$ are negated.
\end{remark}

\end{landscape}

\bibliographystyle{alpha}
%Included for winedt:
%input "bibliography/bibliography.bib"
{\footnotesize{
\bibliography{../../bibliography/bibliography}
}

\end{document}

%% file: preamble.tex
\usepackage{amsmath,amssymb,amsfonts,amsthm}
\usepackage{ifpdf}
\usepackage{enumerate}
\usepackage{ulem}
\usepackage{leftidx}

\usepackage{breqn}

\usepackage{tikz}
\usetikzlibrary{calc}
\usepackage{tikz-qtree}

\ifpdf
\usepackage[pdftex,plainpages=false,hypertexnames=false,pdfpagelabels]{hyperref}
\else
\usepackage[dvips,plainpages=false,hypertexnames=false]{hyperref}
\fi

\newcommand{\arxiv}[1]{\href{http://arxiv.org/abs/#1}{\tt arXiv:\nolinkurl{#1}}}
\newcommand{\doi}[1]{\href{http://dx.doi.org/#1}{{\tt DOI:#1}}}

\newcommand{\googlebooks}[1]{(preview at \href{http://books.google.com/books?id=#1}{google books})}

\theoremstyle{plain}
\newtheorem{prop}{Proposition}[section]
\newtheorem{conj}[prop]{Conjecture}
\newtheorem{thm}[prop]{Theorem}

\newtheorem{lem}[prop]{Lemma}
\newtheorem{cor}[prop]{Corollary}
\newtheorem*{cor*}{Corollary}
\newtheorem*{claim*}{Claim}
\numberwithin{equation}{section}
\theoremstyle{remark}
\newtheorem{ex}[prop]{Example}

\newtheorem{remark}[prop]{Remark}           
\newtheorem*{rem*}{Remark}               %unnumbered remark
\newtheorem*{ex*}{Example}                %unnumbered exercise

\theoremstyle{definition}
\newtheorem{defn}[prop]{Definition}         % numbered definition

\newtheorem{nota}[prop]{Notation}   
\newtheorem*{defn*}{Definition}             % unnumbered definition

\theoremstyle{plain}

\newcounter{comment}
\newcommand{\noop}[1]{}

\def\clap#1{\hbox to 0pt{\hss#1\hss}}

\def\semicolon{;}
\def\applytolist#1{
    \expandafter\def\csname multi#1\endcsname##1{
        \def\multiack{##1}\ifx\multiack\semicolon
            \def\next{\relax}
        \else
            \csname #1\endcsname{##1}
            \def\next{\csname multi#1\endcsname}
        \fi
        \next}
    \csname multi#1\endcsname}

\def\calc#1{\expandafter\def\csname c#1\endcsname{{\mathcal #1}}}
\applytolist{calc}QWERTYUIOPLKJHGFDSAZXCVBNM;
\def\bbc#1{\expandafter\def\csname bb#1\endcsname{{\mathbb #1}}}
\applytolist{bbc}QWERTYUIOPLKJHGFDSAZXCVBNM;
\def\bfc#1{\expandafter\def\csname bf#1\endcsname{{\mathbf #1}}}
\applytolist{bfc}QWERTYUIOPLKJHGFDSAZXCVBNM;

\renewcommand{\imath}{\mathfrak{i}}
\renewcommand{\jmath}{\mathfrak{j}}

\newcommand{\set}[1]{\left\{#1\right\}}

\makeatletter
\newcommand{\hashdef}[2]{\@namedef{#1}{#2}}
\newcommand{\hashlookup}[1]{\@nameuse{#1}}
\makeatother

\IfFileExists{../../graphs/lookup.tex}%
	{\newcommand{\pathtographs}{../../graphs/}}%
	{\newcommand{\pathtographs}{diagrams/graphs/}}

\input{\pathtographs lookup.tex}

\newcommand{\bigraph}[1]{{\hspace{-3pt}\begin{array}{c}%
  \raisebox{-2.5pt}{\includegraphics[height=6mm]{\pathtographs \hashlookup{#1}}}% 
\end{array}\hspace{-3pt}}}

%% file: diagrams/graphs/lookup.tex
\hashdef{bwd1v1v1p1v1x0p1x0p0x1p0x1v0x1x1x0duals1v1v1x3x2x4}{CC3B89144EA94D74}%
\hashdef{bwd1v1v1p1v1x0p1x0p0x1p0x1v1x0x0x1duals1v1v1x3x2x4}{9323526134A1C4D4}%
\hashdef{bwd1v1v1p1v1x1v1v1duals1v1v1v1}{A207DA908F977FE3}%
\hashdef{bwd1v1v1v1p1p1v1x0x0p0x1x0p0x0x1v1x0x0p0x1x0p0x0x1duals1v1v1x2x3v1x2x3}{772136FAF425B059}%
\hashdef{bwd1v1v1v1v1p1p1v1x0x0p0x1x0p0x0x1v1x0x0p0x1x0v1x0p0x1duals1v1v1v2x1x3v2x1}{ACE47192690819DF}%